\documentclass[11pt]{amsart}
\usepackage[T1]{fontenc}
\usepackage{lmodern}
\usepackage{amsmath}
\usepackage{graphicx}
\usepackage{amssymb}
\usepackage{amsthm}
\usepackage{hyperref}
\usepackage{booktabs}
\hypersetup{hidelinks,
  pdftitle={A dual reformulation of the complex sin2-algorithm: exact identities, descent, and finiteness},
  pdfauthor={Ludovic Tagnon},
  pdfsubject={Number theory: structure theory for the complex sin2-algorithm (Problem 4 of Karpenkov)},
  pdfkeywords={Hermite's problem, multidimensional continued fractions, sin2-algorithm, complex cubic fields, conformal module, height descent, finiteness}}

\newtheorem{theorem}{Theorem}[section]
\newtheorem{proposition}[theorem]{Proposition}
\newtheorem{lemma}[theorem]{Lemma}
\newtheorem{corollary}[theorem]{Corollary}

\newtheorem{openlemma}[theorem]{Open Lemma}
\theoremstyle{definition}
\newtheorem{definition}[theorem]{Definition}
\theoremstyle{remark}
\newtheorem{remark}[theorem]{Remark}

\title[Dual reformulation of the complex \texorpdfstring{$\sin^2$}{sin2}-algorithm]{A dual reformulation of the complex \texorpdfstring{$\sin^2$}{sin2}-algorithm:\\ exact identities, descent, and finiteness}
\author{Ludovic Tagnon}
\email{ludovic.tn.ry@gmail.com}
\address{Nancy, France}

\begin{document}

\begin{abstract}
We develop the structure theory of the deterministic $\sin^2$-type algorithm
for complex cubic fields introduced in the companion paper, addressing the
complex-signature case of Karpenkov's Problem~4. The selection rule is shown
to be, exactly, the minimization of a conformal module: the hyperbolic cosine of the
distance between the transverse complex structure of the state and the round
point.
All governing quantities are exact elements of the real embedding of the
field and satisfy closed dual-type identities; in particular no isotropic
candidate ever arises, and the transverse deviation lattice has exactly
pinned covolume. We prove an unconditional soft-rebound lemma (the module can
grow by at most the factor $\varphi^2=2.618\ldots$ in one step), a finiteness theorem for states of bounded
module and height at fixed coordinate discriminant, with explicit static
constants, and a per-field
periodicity theorem under two named hypotheses: (C$_\kappa$), contraction of the
module in the high phase --- partially reduced here to a fixed finite minimax
over a five-parameter compact with rational objective, supported by sampled
adversarial sweeps in exact or 40-digit arithmetic and by a guarded
floating-point Bernstein pass screening $28.9\%$ of the domain volume
(almost entirely its out-of-scope branch; the guard-accepted descent part
is $0.007\%$) ---
and (B), recurrence of bounded height --- which we then \emph{prove}
under (C$_\kappa$) alone: a height-descent theorem (componentwise key-ratio
dichotomy, backward comparison, integral floors) shows the height can
never exceed $\max(H(s_0),C_H)$ with an explicit constant. The remaining
program for per-field periodicity is reduced to (R) on the compact and to
the proved stretched subcases, modulo the aligned stretched subcase, the
remaining $u_2$-road band and subcases, and an open collar-contraction
statement. All proved statements and certificates
are finite and exact.  A machine-checked core of the paper is sealed in
Lean 4, kernel-only, under the standard axioms: the two-case analytic core of
the height-descent theorem, the finiteness pigeonhole, the dual and conformal
identity layer in embedding coordinates, and an abstract assembly theorem
composing them through named interface hypotheses.
\end{abstract}

\maketitle

\section{Introduction}\label{sec:intro}

Hermite's question \cite{He1850} --- find a representation of real numbers
whose eventual periodicity characterizes cubic irrationalities --- has a
modern, precise incarnation in Karpenkov's $\sin^2$-algorithm: for
\emph{totally real} cubic fields, a deterministic Jacobi--Perron type
procedure whose periodicity is proved \cite{Ka22,Ka24}, closing the totally
real case of the problem. For the complementary signature --- one real and
one pair of complex embeddings --- Karpenkov formulates the extension as an
open problem (\cite{Ka24}, \S9, Problem~4).

The companion paper \cite{companion} constructed a deterministic
$\sin^2$-type algorithm for the complex signature (the selection rule flips
sign: the relevant quantity is \emph{minimized}), specified an exact
tie-breaking order, proved the propagation lemma that turns a projective
return into eventual periodicity of the digit sequence, and certified
eventual periodicity on all tested cubic fields, with exact unit
certificates. What it did not provide is a proof mechanism. This paper
supplies the structure theory.

\subsection*{Contributions}

\emph{1. The selection rule is conformal (\S\ref{sec:dual}).} We prove exact
closed identities: the squared norm of the defining cross product is a
discriminant times the conjugate embedding of the dual quadratic element
(Theorem~\ref{thm:dual}); consequently no isotropic candidate ever arises
(Theorem~\ref{thm:noniso}), and minimizing the companion paper's strictly negative score ---
equivalently maximizing its absolute value --- is \emph{exactly} minimizing a
conformal module $m>1$, the hyperbolic cosine of the distance from the
transverse complex structure to the round point
(Propositions~\ref{prop:conformal}--\ref{prop:hyperbolic}). All these
quantities live in the real embedding of the field: every comparison is
exact (Proposition~\ref{prop:exact}).

\emph{2. The geometry of moves (\S\ref{sec:geom}).} Moves act on slopes as
real affine barycenters; periodic orbits are self-similar under an explicit
unit spiral; the transverse deviation lattice has \emph{exactly} computable
covolume $\delta x_1/(2\lvert\xi\rvert^2)$ (Theorem~\ref{thm:covol}), which
by Minkowski always contains a short direction --- the structural source of
contraction; recentering obeys a Pythagoras identity with an explicit drift
functional; the real coordinates are non-increasing.

\emph{3. Descent and the conditional theorem (\S\ref{sec:descent}).} An
unconditional \emph{soft rebound} lemma (one step can never grow $m$ by more
than the factor $\varphi^2$, Lemma~\ref{lem:soft}); a combinatorial descent lemma; and the main
structural result: under (C$_\kappa$) --- contraction of $m$ in the high
phase, partially reduced in \S\ref{sec:lock} to a finite certifiable
minimax plus explicitly marked deferred subcases --- every orbit is eventually periodic with a unit certificate:
Theorem~\ref{thm:main} proves this under the auxiliary height-recurrence
hypothesis (B), and Theorem~\ref{thm:hdescent} \emph{proves (B)} from
(C$_\kappa$) via a height-descent argument with explicit constants.

\emph{4. Finiteness (\S\ref{sec:finiteness}).} States of bounded module and
height, at fixed coordinate discriminant, form finitely many unit classes,
with explicit static bounds
(Theorem~\ref{thm:finiteness}).

\emph{5. The case machine and the lock (\S\ref{sec:cases}--\ref{sec:lock}).}
Quantified witnesses for (C$_\kappa$) in proved stretched subcases
(contraction $21.2/r$ and a restricted $u_2$-road via the exact covolume);
the composed-lever mechanism; an integral gap lemma fencing the degenerate
boundaries ($x_1/x_2,\,x_2/x_3\ge1+1/(4B)$); and a partial reduction of the
remaining compact region to a \emph{finite certifiable minimax} over five
parameters, with sampled numerical evidence and a partial interval
certification. The aligned stretched regime, the remaining $u_2$-road
subcases, and the quantitative collar contraction are stated as
deferred/open.

\emph{6. Certificates and formalization (\S\ref{sec:formal}).} Every link of
the chain is a finite exact statement; a machine-checked core of the paper is
sealed in Lean 4, kernel-only, continuing the companion paper's discipline ---
\S\ref{sec:formal} states its exact scope and the remaining formal work.

In short: the structural part of Problem 4 per field is reduced to a small
set of explicit finite or local tasks: (R) on the compact, the aligned
stretched subcase, the remaining $u_2$-road band and subcases, and the open
collar-contraction lemma. The height
recurrence (B), previously a separate obstruction, is proved from
(C$_\kappa$).

\medskip
\textbf{Setting and notation.} Throughout, $K=\mathbb{Q}(\alpha)$ is a cubic
field of signature $(1,1)$, with real embedding $\sigma_r$ and one complex
embedding $\sigma_c$; $L\subset\mathcal{O}_K$ is an order (in the applications
$L=\mathbb{Z}[\alpha]$). A \emph{state} is an ordered triple
$s=(u_1,u_2,u_3)$ of elements of $L$ forming a $\mathbb{Q}$-basis of $K$,
\emph{admissible} if $\sigma_r(u_i)>0$ for all $i$, and sorted so that
$x_1\ge x_2\ge x_3>0$ where $x_i:=\sigma_r(u_i)$. We write
$\xi:=(x_1,x_2,x_3)\in\mathbb{R}^3$ and $\nu:=(\sigma_c(u_1),\sigma_c(u_2),
\sigma_c(u_3))\in\mathbb{C}^3$, and $\nu=a+ib$ with $a,b\in\mathbb{R}^3$.
The moves of the algorithm are the unimodular matrices $V(A,B,g)$
(the letter $B$ is overloaded by local conventions --- the move coefficient
here, the height bound of the finiteness statements, and Hypothesis~(B) are
unrelated, each use local to its statement)
($A,B,g\in\mathbb{Z}_{\ge0}$) and $W$ of the
companion paper, acting simultaneously on $\xi$, $a$, $b$; the window
constraints are positivity constraints on the new real coordinates. The
quantity $\delta:=2\,\lvert\det(\xi,a,b)\rvert=\lvert\mathrm{disc}(u_1,u_2,u_3)
\rvert^{1/2}$ is invariant along every orbit. Cross and dot products on
$\mathbb{C}^3$ are \emph{bilinear} (non-Hermitian).

\section{The dual and conformal framework}\label{sec:dual}

\subsection{The dual identity}

For a state $s$, let $(u_1^\vee,u_2^\vee,u_3^\vee)$ denote the trace-dual basis
of $(u_1,u_2,u_3)$, i.e.\ $\mathrm{Tr}_{K/\mathbb{Q}}(u_iu_j^\vee)=\delta_{ij}$,
and set
\[
Q:=u_1^2+u_2^2+u_3^2\in K,\qquad
Q^\vee:=(u_1^\vee)^2+(u_2^\vee)^2+(u_3^\vee)^2\in K .
\]

\begin{theorem}[dual identity]\label{thm:dual}
Let $n_1=\xi\times\nu$. Then
\[
\langle n_1,n_1\rangle \;=\; \mathrm{disc}(u_1,u_2,u_3)\cdot
\overline{\sigma_c}\!\left(Q^\vee\right),
\]
and consequently the score of the algorithm satisfies
\[
\lvert\mathrm{score}(s)\rvert \;=\;
\frac{\sigma_r(Q)}{\lvert\mathrm{disc}(u)\rvert\cdot\lvert\sigma_c(Q^\vee)\rvert^{2}} .
\]
\end{theorem}

\begin{proof}
Let $M\in\mathbb{C}^{3\times 3}$ be the embedding matrix $M_{i\sigma}=\sigma(u_i)$,
$\sigma\in\{\sigma_r,\sigma_c,\overline{\sigma_c}\}$. By the bilinear Lagrange
identity, $\langle\xi\times\nu,\xi\times\nu\rangle=B_{rr}B_{cc}-B_{rc}^2$ where
$B_{\sigma\tau}=\sum_i\sigma(u_i)\tau(u_i)=(M^{\mathsf T}M)_{\sigma\tau}$. This
is the $(\bar c,\bar c)$ cofactor of the Gram matrix $M^{\mathsf T}M$, hence
equals $\det(M)^2\cdot\big[(M^{\mathsf T}M)^{-1}\big]_{\bar c\bar c}
=\det(M)^2\sum_i (M^{-1})_{\bar c,i}^2$. Trace duality reads
$M\cdot\big(\sigma(u_j^\vee)\big)^{\mathsf T}=I$, so the rows of $M^{-1}$ are
the embeddings of the dual basis: $(M^{-1})_{\sigma,j}=\sigma(u_j^\vee)$.
Since $\overline{\sigma_c}$ is a field morphism,
$\sum_i\overline{\sigma_c}(u_i^\vee)^2=\overline{\sigma_c}(Q^\vee)$, and
$\det(M)^2=\mathrm{disc}(u)$. The score formula follows from the closed form of
the companion paper, $\lvert\mathrm{score}\rvert=\lvert\det(\xi,\nu,\bar\nu)
\rvert^2\lvert\xi\rvert^2/\lvert\langle n_1,n_1\rangle\rvert^2$, together with
$\lvert\det(\xi,\nu,\bar\nu)\rvert^2=\lvert\mathrm{disc}(u)\rvert$ and
$\lvert\xi\rvert^2=\sigma_r(Q)$.
\end{proof}

\subsection{Totality: no isotropic candidate exists}

\begin{theorem}[non-isotropy]\label{thm:noniso}
For every triple $u$ forming a $\mathbb{Q}$-basis of $K$,
$\langle \xi\times\nu,\xi\times\nu\rangle\neq 0$. In particular the isotropy
exclusion in the algorithm's candidate set never fires: the enumeration is
structurally total.
\end{theorem}

\begin{proof}
$\sigma_r(Q^\vee)=\sum_i\sigma_r(u_i^\vee)^2$ is a sum of three real squares; it
vanishes only if the $\sigma_r$-row of $M^{-1}$ vanishes, contradicting
invertibility. Hence $Q^\vee\neq 0$ in $K$, and since $\overline{\sigma_c}$ is
injective, $\overline{\sigma_c}(Q^\vee)\neq 0$; also $\mathrm{disc}(u)\neq 0$.
Apply Theorem~\ref{thm:dual}.
\end{proof}

\subsection{The conformal module}

Write $\hat\xi=\xi/\lvert\xi\rvert$, $\alpha_\parallel=\langle a,\hat\xi\rangle$,
$\beta_\parallel=\langle b,\hat\xi\rangle$, $a_\perp=a-\alpha_\parallel\hat\xi$,
$b_\perp=b-\beta_\parallel\hat\xi$, and
\[
T:=\lvert a_\perp\rvert^2+\lvert b_\perp\rvert^2,\qquad
A_\perp:=\lvert a_\perp\times b_\perp\rvert,\qquad
m(s):=\frac{T}{2A_\perp}\;\ge 1 .
\]

\begin{figure}[htbp]
\centering
\includegraphics[width=0.62\textwidth]{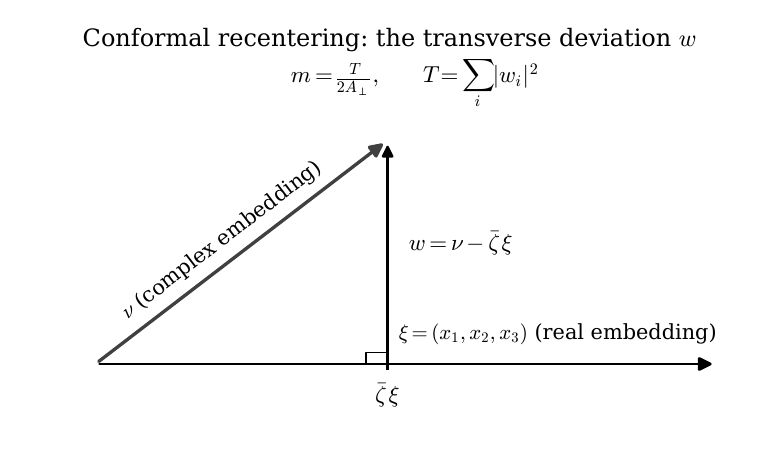}
\caption{The transverse deviation.  The complex embedding vector $\nu$ is
recentered by removing its projection $\bar\zeta\,\xi$ onto the real
embedding direction $\xi$, leaving the transverse part $w=\nu-\bar\zeta\,
\xi\perp\xi$.  The conformal module $m=T/(2A_\perp)$, with
$T=\sum_i|w_i|^2$, measures how far the transverse complex structure is
from round; the selection rule minimizes it exactly
(Proposition~\ref{prop:conformal}).}
\label{fig:recentering}
\end{figure}

\begin{proposition}[the selection rule is conformal]\label{prop:conformal}
For every candidate $s'$,
\[
\lvert\mathrm{score}(s')\rvert\;=\;\frac{1}{m(s')^2-1}\;,
\]
an exact identity with no constants. Hence selecting the most negative score
is \emph{exactly} minimizing the conformal module $m$: the deterministic
algorithm selects, at every step, the candidate whose transverse complex
structure is closest to round.
\end{proposition}

\begin{proof}
Write $n_1=\xi\times\nu=\xi\times a+i\,\xi\times b$, and let $G$ be the Gram
matrix of $(\xi\times a,\xi\times b)$. Then
$\lvert\langle n_1,n_1\rangle\rvert^2=(\lvert\xi\times a\rvert^2-\lvert\xi
\times b\rvert^2)^2+4\langle\xi\times a,\xi\times b\rangle^2=
(\operatorname{tr}G)^2-4\det G$. Parallel components drop out of cross
products, so $\operatorname{tr}G=\lvert\xi\rvert^2(\lvert a_\perp\rvert^2+
\lvert b_\perp\rvert^2)=\lvert\xi\rvert^2T$, while by Lagrange
$\det G=\lvert(\xi\times a)\times(\xi\times b)\rvert^2=\lvert\xi\rvert^2
\det(\xi,a,b)^2=\lvert\xi\rvert^4A_\perp^2$. Hence
$\lvert\langle n_1,n_1\rangle\rvert^2=\lvert\xi\rvert^4(T^2-4A_\perp^2)$.
The closed form of the score (companion paper) and
$\lvert\det(\xi,\nu,\bar\nu)\rvert^2=4\det(\xi,a,b)^2=4\lvert\xi\rvert^2
A_\perp^2$ give
\[
\lvert\mathrm{score}(s')\rvert
=\frac{\lvert\det(\xi,\nu,\bar\nu)\rvert^2\lvert\xi\rvert^2}
{\lvert\langle n_1,n_1\rangle\rvert^2}
=\frac{4\lvert\xi\rvert^4A_\perp^2}{\lvert\xi\rvert^4(T^2-4A_\perp^2)}
=\frac{1}{(T/2A_\perp)^2-1}\;. \qedhere
\]
\end{proof}

\begin{proposition}[hyperbolic reading]\label{prop:hyperbolic}
Let $z_\perp:=(\langle a_\perp,b_\perp\rangle+iA_\perp)/\lvert a_\perp\rvert^2
\in\mathbb{H}$. Then $m(s)=\cosh d_{\mathbb H}(z_\perp,i)$. The algorithm is a
minimizing walk in the hyperbolic plane; multiplication of the state by a unit
$\lambda$ acts on the pair $(a_\perp,b_\perp)$ by a complex similarity, whose
rotation part moves $z_\perp$ by the elliptic isometry of $\mathbb H$ fixing
$i$ (through twice the argument of $\sigma_c(\lambda)$, up to orientation); in
particular $d_{\mathbb H}(z_\perp,i)$ --- hence $m$ --- is unit-invariant,
while the marked point $z_\perp$ itself is fixed only when
$\sigma_c(\lambda)$ is real. (For actual states the rotation genuinely
moves $z_\perp$: non-isotropy gives $m>1$, hence $z_\perp\ne i$.)
\end{proposition}

\begin{proof}
The two-dimensional Lagrange identity
$\langle u,v\rangle^2+\det(u,v)^2=\lvert u\rvert^2\lvert v\rvert^2$ gives, for
$z=x+iy$ as defined, $(x^2+y^2+1)/(2y)=(\lvert a_\perp\rvert^2+\lvert
b_\perp\rvert^2)/(2A_\perp)=m$, and $\cosh d_{\mathbb H}(z,i)=(x^2+y^2+1)/(2y)$.
For the unit action, $u\mapsto\lambda u$ multiplies $\nu$ by
$\sigma_c(\lambda)$ and $\xi$ by $\sigma_r(\lambda)$, hence multiplies the
complex vector $w=a_\perp+ib_\perp$ by $\sigma_c(\lambda)$. Writing
$\sigma_c(\lambda)=\rho e^{i\theta}$, the modulus scales $T$ and $A_\perp$
by $\rho^2$ (so $m$ is unchanged), while the rotation part mixes the frame:
$a_\perp'=\rho(\cos\theta\,a_\perp-\sin\theta\,b_\perp)$,
$b_\perp'=\rho(\sin\theta\,a_\perp+\cos\theta\,b_\perp)$. The
quantities $\lvert a_\perp\rvert^2$ and
$\langle a_\perp,b_\perp\rangle$ are therefore not individually scaled when
$\theta\ne0$, and the M\"obius action of the rotation on
$z_\perp$ is the elliptic isometry fixing $i$: the distance to $i$, hence
$m=\cosh d_{\mathbb H}(z_\perp,i)$, is preserved, while $z_\perp$ moves
along the hyperbolic circle of that radius about $i$.
\end{proof}

\begin{proposition}[deviation form]\label{prop:pseudo}
Let $\bar\zeta:=\langle\xi,\nu\rangle/\lvert\xi\rvert^2$ and
$w_i:=\sigma_c(u_i)-\bar\zeta\,x_i$, so that $\sum_i x_iw_i=0$. Then
$\nu_\perp=(w_1,w_2,w_3)$, and with $T=\sum\lvert w_i\rvert^2$,
$V:=\sum w_i^2$:
\[
4A_\perp^2=T^2-\lvert V\rvert^2, \qquad
m=\frac{1}{\sqrt{1-\rho^2}},\quad \rho:=\frac{\lvert V\rvert}{T}\in[0,1).
\]
Moreover $m=1$ iff $V=0$ (balanced configuration of deviations).
\end{proposition}

\begin{proof}
The identity $\lvert\langle v,v\rangle\rvert^2=T^2-4A^2$ for a complex vector
$v$ with real and imaginary parts $(a_\perp,b_\perp)$ is the same Gram
computation as above, applied to $\nu_\perp$; the display follows.
\end{proof}

\begin{proposition}[exactness]\label{prop:exact}
The quantities $\mathrm{score}(s)$, $m(s)^2$ and
$P(s):=\sigma_r(Q)\,\sigma_r(Q^\vee)$ all lie in $\sigma_r(K)$; on integral
states all comparisons between them are decidable exactly.
\end{proposition}

\begin{proof}
Each is invariant under the transposition of the two complex places, whose
fixed field inside the Galois closure is $\sigma_r(K)$.
\end{proof}

\begin{remark}
$P\ge 1$ always (Cauchy--Schwarz against $\sum_i\sigma_r(u_i)\sigma_r(u_i^\vee)
=(M^{-1}M)_{rr}=1$), $P$ is scaling-invariant, and $P=\sec^2\theta$ where
$\theta$ is the angle between $\xi$ and the normal to the plane
$\mathrm{span}(a,b)$; since $\delta=2\lvert\langle\xi,a\times b\rangle\rvert$ is
pinned, $P\le C$ bounds the \emph{product} $\lvert\xi\rvert\cdot\lvert a\times
b\rvert$. These facts are used in the descent analysis of \S4.
\end{remark}

\subsection{Unification of the signatures}

\begin{proposition}\label{prop:unif}
In the totally real case, with $\eta_2,\eta_3$ the second and third real
embedding vectors, the same Gram computation yields
\[
\sin^2\alpha=\frac{\det(\xi,\eta_2,\eta_3)^2\,\lvert\xi\rvert^2}{\lvert\xi\times\eta_2\rvert^2\,\lvert\xi\times\eta_3\rvert^2}:
\]
maximizing
$\sin^2$ is minimizing $\lvert\xi\rvert\cdot\lvert\eta_{2,\perp}\rvert\cdot
\lvert\eta_{3,\perp}\rvert$. The two signatures thus run \emph{one} rule ---
minimize the transverse size of the other places relative to the real one ---
the complex case replacing the separated product by the $U(1)$-invariant
conformal module.
\end{proposition}

\begin{proof}
Apply the Binet--Cauchy identity of Theorem~\ref{thm:covol} to each of the
two real pairs $(\xi,\eta_2)$, $(\xi,\eta_3)$ separately and take the
product; the mixed Gram terms cancel exactly as in
Proposition~\ref{prop:conformal}, leaving the stated quotient.
\end{proof}

\section{The geometry of moves}\label{sec:geom}

\subsection{Slopes and the unit spiral}

Define the \emph{slope} of $u\in K^\times$ as $\zeta(u):=\sigma_c(u)/\sigma_r(u)
\in\mathbb{C}$, and the slope cloud of a state as $(\zeta_1,\zeta_2,\zeta_3)$.

\begin{proposition}[moves are real affine barycenters]\label{prop:bary}
Let $V=(v_{ij})\in\mathrm{GL}_3(\mathbb{Z})$ be an admissible move, producing
$u_i'=\sum_j v_{ij}u_j$. Then
\[
\zeta(u_i')=\sum_j \pi_{ij}\,\zeta_j,\qquad
\pi_{ij}:=\frac{v_{ij}x_j}{x_i'},\qquad \sum_j\pi_{ij}=1,\ \pi_{ij}\in\mathbb{R}.
\]
Each new slope is an affine (in general non-convex) barycenter of the old
slopes with \emph{real} weights; the window constraints control only the
denominators $x_i'>0$.
\end{proposition}

\begin{proof}
$\sigma_c(u_i')=\sum_j v_{ij}\sigma_c(u_j)=\sum_j v_{ij}x_j\zeta_j$ and
$x_i'=\sum_j v_{ij}x_j$.
\end{proof}

\begin{proposition}[unit spiral]\label{prop:spiral}
Suppose an orbit is eventually periodic: $s^{(t+p)}=\lambda\cdot s^{(t)}$ for
some unit $\lambda\in\mathcal{O}_K^\times$ with $\sigma_r(\lambda)>0$ and all
large $t$. Then the slope clouds satisfy
$\zeta^{(t+p)}=\mu\,\zeta^{(t)}$ with
\[
\mu:=\frac{\sigma_c(\lambda)}{\sigma_r(\lambda)},\qquad
\lvert\mu\rvert=\sigma_r(\lambda)^{-3/2},\qquad
\arg\mu=\arg\sigma_c(\lambda),
\]
i.e.\ the cycle is a $\mu$-self-similar configuration under a fixed logarithmic
spiral, whose rotation number is the argument of the unit at the complex place.
\end{proposition}

\begin{proof}
$\zeta(\lambda u)=\mu\,\zeta(u)$ directly; $\lvert N(\lambda)\rvert=
\sigma_r(\lambda)\lvert\sigma_c(\lambda)\rvert^2=1$ gives
$\lvert\sigma_c(\lambda)\rvert=\sigma_r(\lambda)^{-1/2}$.
\end{proof}

\subsection{The exact covolume package}

Recall the deviations $w_i=\sigma_c(u_i)-\bar\zeta x_i$ of
Proposition~\ref{prop:pseudo}, with $\sum_i x_iw_i=0$.

\begin{theorem}[exact transverse areas]\label{thm:covol}
Write $P_{ij}:=\operatorname{Im}(\bar w_iw_j)$ (the oriented area of the
parallelogram spanned by $w_i,w_j$ in $\mathbb{C}$). Then:
\begin{enumerate}
\item[(a)] (Binet--Cauchy) $A_\perp^2=P_{12}^2+P_{13}^2+P_{23}^2$;
\item[(b)] the constraint $\sum x_iw_i=0$ forces the exact proportionality
\[
(P_{23},\,-P_{13},\,P_{12})=c\,(x_1,\,x_2,\,x_3),\qquad
\lvert c\rvert=\frac{A_\perp}{\lvert\xi\rvert}=\frac{\delta}{2\lvert\xi\rvert^{2}};
\]
\item[(c)] consequently the plane lattice $\mathbb{Z}w_2+\mathbb{Z}w_3\subset
\mathbb{C}$ has covolume exactly
\[
\operatorname{covol}(\mathbb{Z}w_2+\mathbb{Z}w_3)=\lvert P_{23}\rvert
=\frac{\delta\,x_1}{2\lvert\xi\rvert^{2}},
\]
and cyclically for the other two pairs.
\end{enumerate}
\end{theorem}

\begin{proof}
(a) is Binet--Cauchy for the $3\times2$ real matrix with columns
$\operatorname{Re}w$, $\operatorname{Im}w$: the squared norm of the cross
product of the columns is the sum of the squared $2\times2$ minors, and the
$(i,j)$ minor is $\operatorname{Re}w_i\operatorname{Im}w_j-
\operatorname{Re}w_j\operatorname{Im}w_i=\operatorname{Im}(\bar w_iw_j)$.
(b) From $w_1=-(x_2w_2+x_3w_3)/x_1$: $\operatorname{Im}(\bar w_1w_2)=
(x_3/x_1)P_{23}$ and $\operatorname{Im}(\bar w_1w_3)=-(x_2/x_1)P_{23}$, so the
area vector is proportional to $(x_1,x_2,x_3)$ up to signs as displayed; its
norm is $A_\perp$ by (a), which fixes $\lvert c\rvert=A_\perp/\lvert\xi\rvert$,
and $A_\perp=\delta/(2\lvert\xi\rvert)$.
(c) The covolume of a rank-2 lattice in $\mathbb{C}$ spanned by $w_2,w_3$ is
$\lvert\operatorname{Im}(\bar w_2w_3)\rvert$.
\end{proof}

\begin{corollary}[Minkowski input]\label{cor:mink}
$\mathbb{Z}w_2+\mathbb{Z}w_3$ contains a nonzero vector $w$ with
$\lvert w\rvert^2\le \frac{2}{\sqrt3}\cdot\frac{\delta x_1}{2\lvert\xi\rvert^2}
=\frac{\delta x_1}{\sqrt3\,\lvert\xi\rvert^2}$. Since $T=m\delta/\lvert\xi
\rvert$, this vector captures at most a fraction $\frac{x_1}{\sqrt3\,m\,
\lvert\xi\rvert}\le\frac{1}{\sqrt3\,m}$ of the total transverse energy: in the
high-$m$ regime the lattice of deviations always contains a direction
\emph{much shorter} than the energy scale --- the structural source of the
contractions observed and used in \S6.
\end{corollary}

\subsection{Recentering: the Pythagoras identity}

After a move $s\mapsto s'=Vs$, the deviations must be recomputed with respect
to the \emph{new} Rayleigh center $\bar\zeta'=\langle\xi',\nu'\rangle/
\lvert\xi'\rvert^2$.

\begin{proposition}[recentering]\label{prop:recenter}
Let $W_i'=\sigma_c(u_i')-\bar\zeta\,x_i'$ be the deviations of $s'$ computed at
the \emph{old} center, and $T'_{\mathrm{old}}=\sum\lvert W_i'\rvert^2$. Then
\[
T'\;=\;\min_{c\in\mathbb{C}}\sum_i\bigl|\sigma_c(u_i')-c\,x_i'\bigr|^2
\;=\;T'_{\mathrm{old}}-\lvert\Delta\bar\zeta\rvert^{2}\lvert\xi'\rvert^{2},
\qquad
\Delta\bar\zeta=\frac{\langle\xi',\,V\nu_\perp\rangle}{\lvert\xi'\rvert^{2}},
\]
where the pairing in the numerator is the real--bilinear dot product applied
entrywise to the complex vector $V\nu_\perp$. In particular
$T'\le T'_{\mathrm{old}}$ always: recentering can only help, and the exact
drift $\Delta\bar\zeta=\bar\zeta'-\bar\zeta$ is a linear functional of the old
transverse data $\nu_\perp$ alone.
\end{proposition}

\begin{proof}
The map $c\mapsto\sum_i\lvert\sigma_c(u_i')-cx_i'\rvert^2$ is a least-squares
problem in $c$; its minimizer is the regression coefficient
$\bar\zeta'$, and the Pythagorean defect of moving the center by
$\Delta\bar\zeta$ is $\lvert\Delta\bar\zeta\rvert^2\lvert\xi'\rvert^2$. For the
drift formula, $\nu'=V\nu=\bar\zeta\,V\xi+V\nu_\perp=\bar\zeta\,\xi'+V\nu_\perp$,
so $\bar\zeta'=\bar\zeta+\langle\xi',V\nu_\perp\rangle/\lvert\xi'\rvert^2$.
\end{proof}

\subsection{Height control and monotonicity}

\begin{proposition}[coordinatewise monotonicity]\label{prop:mono}
For every admissible move of the algorithm, each real coordinate is
non-increasing: $x_i'\le x_i$ for all $i$ (same position). In particular
$\max_i x_i$ is non-increasing along every orbit.
\end{proposition}

\begin{proof}
For $V(A,B,g)$ acting on $(x_1,x_2,x_3)$: the window constraints include
$x_2-Bx_3\ge0$ and the positivity chain of the companion paper, and one
rewrites $x_1'=x_1-g(x_2-Bx_3)-Ax_3\le x_1$, $x_2'=x_2-Bx_3\le x_2$,
$x_3'=x_3$. For $W$: $x_1'=x_1-x_2\le x_1$, $x_2'=x_2$,
$x_3'=x_3-(x_1-x_2)\le x_3$, the window requiring $x_1-x_2\ge 0$ resp.\ the
displayed positivity.
\end{proof}

\begin{proposition}[height decomposition]\label{prop:height}
With $H(s):=\max_i\lvert N(u_i)\rvert$ and $\bar\zeta$ the Rayleigh center,
\[
H(s)\;\le\;2\,\lvert\bar\zeta\rvert^{2}\,x_1^{3}
\;+\;2\,m\,\delta\;\le\;2\,\lvert\bar\zeta\rvert^{2}\,
\lvert\xi\rvert^{3}\;+\;2\,m\,\delta .
\]
Thus along any portion of orbit where $\lvert\bar\zeta\rvert^2\lvert\xi\rvert^3$
and $m$ stay bounded, the arithmetic height stays bounded --- the bridge
between the analytic descent of \S4 and the finiteness theorem of \S5.
\end{proposition}

\begin{proof}
$\lvert N(u_i)\rvert=x_i\lvert\sigma_c(u_i)\rvert^2=
x_i\lvert\bar\zeta x_i+w_i\rvert^2\le
2x_i(\lvert\bar\zeta\rvert^2x_i^2+\lvert w_i\rvert^2)\le
2\lvert\bar\zeta\rvert^2x_1^3+2x_1T$, and $x_1T\le\lvert\xi\rvert\cdot
m\delta/\lvert\xi\rvert=m\delta$, $x_1\le\lvert\xi\rvert$.
\end{proof}

\section{Descent and the conditional theorem}\label{sec:descent}

\begin{figure}[htbp]
\centering
\includegraphics[width=0.66\textwidth]{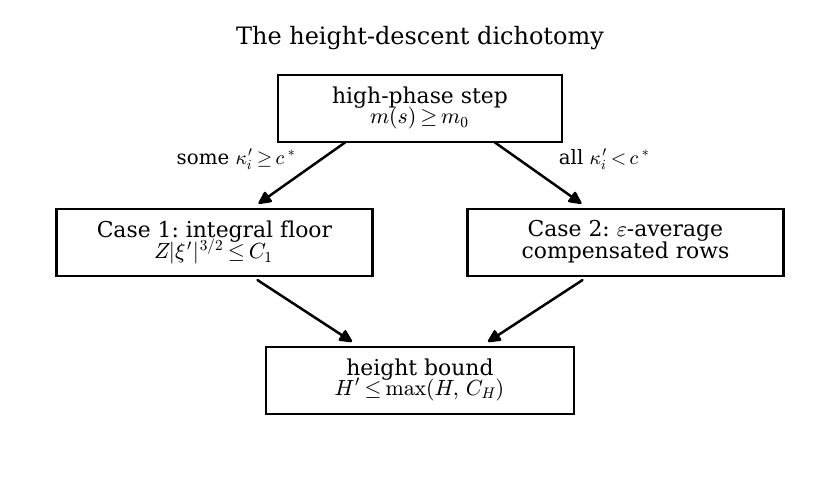}
\caption{The height-descent dichotomy proving (B) from (C$_\kappa$).  A
high-phase step splits on whether some component key-ratio $\kappa_i'$
reaches the threshold $c^*$: Case~1 (integral floor) bounds
$Z|\xi'|^{3/2}$ by an explicit constant, Case~2 ($\varepsilon$-average of
the compensated rows) gives componentwise height descent; either way the
height stays below $\max(H,C_H)$.}
\label{fig:height-descent}
\end{figure}

\subsection{Combinatorial descent}

Fix $\beta\in(0,1)$ and set $\Phi(s):=\log m(s)+\beta\log r(s)$, where
$r=x_1/x_3\ge1$ is the real aspect ratio.

\begin{lemma}[combinatorial descent]\label{lem:D1}
Let $(s_t)$ be an orbit and suppose there are constants $m_0\ge1$, $C\ge1$,
$c_0>0$, $K\in\mathbb{N}$ such that whenever $m(s_t)\ge m_0$:
\begin{enumerate}
\item[(H1)] $\Phi(s_{t+1})\le\Phi(s_t)+\log C$;
\item[(H2)] there exists $k=k(t)\in\{1,\dots,K\}$ with
$\Phi(s_{t+k})\le\Phi(s_t)-c_0$.
\end{enumerate}
Let $t^*=\inf\{t:m(s_t)<m_0\}$. Then
\[
t^*\;\le\;K\,\frac{\Phi(s_0)-\log m_0}{c_0}+K,
\qquad
\sup_{t\le t^*}\log m(s_t)\;\le\;\Phi(s_0)+K\log C .
\]
\end{lemma}

\begin{proof}
Define $t_0=0$ and $t_{j+1}=t_j+k(t_j)$ as long as $m\ge m_0$ at $t_j$. By
(H2), $\Phi(s_{t_j})\le\Phi(s_0)-jc_0$. Since $r\ge1$ gives $\beta\log r\ge0$,
we have $\Phi\ge\log m\ge\log m_0$ while the phase lasts; hence
$j\le(\Phi(s_0)-\log m_0)/c_0$ blocks, and $t^*$ exceeds the last $t_j$ by at
most $K$. Inside a block, (H1) applied at most $K$ times yields
$\Phi(s_t)\le\Phi(s_{t_j})+K\log C\le\Phi(s_0)+K\log C$, and $\log m\le\Phi$.
\end{proof}

\subsection{The core hypothesis, the soft rebound, and height recurrence}

\begin{definition}[core contraction]\label{hyp:core}
Say $L$ satisfies \emph{Hypothesis (C$_\kappa$)} if there exist $\gamma>1$,
$m_0\ge1$ and a step bound $\kappa\in\mathbb{N}$ (the subscript of the
hypothesis' name; $\kappa$ is a step count, unrelated to the field $K$)
such that from every admissible state $s$ of
$L$ with $m(s)\ge m_0$, the orbit of the algorithm itself --- the successive
min-$m$ selections, not an auxiliary witness path --- reaches
$m\le\max(m(s)/\gamma,\,m_0)$ within at most $\kappa$ steps. The notation
(C$_1$) denotes the one-step form.
\end{definition}

\begin{remark}[why the dynamic form]\label{rem:greedy}
``Some admissible chain of length $k$ contracts'' transfers to the real orbit
only for $k=1$, where the selection is by definition at least as good as any
witness candidate; for $k\ge2$ the greedy orbit leaves the witness path at
the first step. Certificates must therefore either be one-step, or certify
the greedy orbit itself over each box (the selection is locally constant on
sufficiently small boxes, by exact separation of scores; at algebraic ties,
all tied branches are certified). The empirical record strongly supports the
one-step form away from the collars: every observed high-phase step of the
real dynamics contracts ($370/370$; smallest observed contraction $m/m'=1.37$, and $1.40$ among the $23$ steps with $t\ge3$), and the hard
truncated-compact configurations already descend at $k=1$ (worst ratio $m'/m=0.876$);
multi-step blocks appear to be needed only near the collars.
\end{remark}

\begin{lemma}[soft rebound, unconditional]\label{lem:soft}
For every integral admissible state $s$, the selected move satisfies
\[
m(s')\;\le\;\varphi^2\,m(s),\qquad \varphi^2=\frac{3+\sqrt5}{2}=2.618\ldots,
\]
and the associated old-center energy estimate
$T'_{\mathrm{old}}\le\varphi^2T$ is approached by Fibonacci-type profiles.
\end{lemma}

\begin{proof}
Integrality forces $x_1>x_2$ strictly (as in Lemma~\ref{lem:gap}), so the
head-difference candidate $V(0,0,1)$, with new head $u_1-u_2$, is admissible.
At the old center its transverse energy is
$T'\le\lvert w_1-w_2\rvert^2+\lvert w_2\rvert^2+\lvert w_3\rvert^2$, a
quadratic form in $(w_1,w_2,w_3)$ whose matrix is
$\bigl(\begin{smallmatrix}1&-1\\-1&2\end{smallmatrix}\bigr)\oplus(1)$ on each
real coordinate; its largest eigenvalue solves $\lambda^2-3\lambda+1=0$,
i.e.\ $\lambda_{\max}=(3+\sqrt5)/2=\varphi^2$. Hence
$T'\le\varphi^2T$ (dropping the constraint $\sum x_iw_i=0$ only enlarges the
supremum), recentering only helps (Proposition~\ref{prop:recenter}), and
$\lvert\xi'\rvert\le\lvert\xi\rvert$ coordinatewise
(Proposition~\ref{prop:mono}); so $m(\text{candidate})\le\varphi^2m$, and the
min-$m$ selection does at least as well. The constant in the old-center
energy estimate is optimal for this quadratic form: writing
$T'_0:=\lvert w_1-w_2\rvert^2+\lvert w_2\rvert^2+\lvert
w_3\rvert^2$, one has the polynomial identity in
$\mathbb{Q}(\sqrt5)$ (componentwise in $\mathrm{Re},\mathrm{Im}$)
\[
\varphi^{3}T-\varphi T'_0=\lvert\varphi w_1+w_2\rvert^{2}+\varphi^{2}\lvert w_3\rvert^{2},
\]
verified by expanding with $\varphi^2=\varphi+1$ (so
$\varphi^3-\varphi=\varphi^2$ and $\varphi^3-2\varphi=1$). Both terms
vanish exactly on the ray $w_2=-\varphi w_1$, $w_3=0$; the transversality
constraint $\sum x_iw_i=0$ then forces $x_1/x_2=\varphi$, which integral
states approach through consecutive-Fibonacci head ratios --- whence the
energy constant is approached, not attained, by Fibonacci-type profiles.
This is not an optimality statement for the actual module rebound: the
recentring term and the factor $\lvert\xi'\rvert/\lvert\xi\rvert$ can only
lower the selected module. The archived case-machine optimization observes
a worst module rebound about $2.081$, strictly below $\varphi^2$; this
number is empirical and is not used in the proof.
\end{proof}

\begin{definition}[height recurrence]\label{hyp:B}
Say an orbit $(s_t)$ satisfies \emph{Hypothesis (B)} if
$\liminf_{t\to\infty}H(s_t)<\infty$, i.e.\ some height bound $B$ is attained
at infinitely many times.
\end{definition}

Hypothesis (C$_\kappa$) is the quantitative engine (\S6--7 reduce its compact
part to a fixed certifiable minimax and isolate the remaining stretched and
collar subcases); Hypothesis (B) rules out
escape of the one mode that the selection cannot see
(Remark~\ref{rem:zetabar}) --- and is \emph{proved} from (C$_\kappa$) in
Theorem~\ref{thm:hdescent} below, so that it never needs to be assumed
independently. Everything else is unconditional.

\begin{theorem}[periodicity, per field]\label{thm:main}
Assume $L$ satisfies (C$_\kappa$) for some step bound $\kappa$, and let $s_0$ be an admissible
initial state whose orbit satisfies (B). Then the orbit is eventually periodic: there
exist $t_0$, $P\ge1$ and a unit $\lambda\in\mathcal{O}_K^\times$ with
$\sigma_r(\lambda)>0$ such that $s_{t+P}=\lambda\,s_t$ for all $t\ge t_0$,
and the digit sequence is eventually $P$-periodic.
\end{theorem}

\begin{proof}
\emph{Step 1 (uniform module bound).} Set
$m^*:=\varphi^{2\kappa}\max(m(s_0),\,m_0)$. By induction on blocks: a
(C$_\kappa$)-block launched at a state of module $\mu\ge m_0$ has at most
$\kappa-1$ nonterminal steps, so Lemma~\ref{lem:soft} caps its
intermediate excursion by $\varphi^{2(\kappa-1)}\mu$, and it terminates at
$\max(\mu/\gamma,m_0)\le\mu$; below $m_0$, one rebound step reaches at
most $\varphi^{2}m_0$. Block starts therefore never exceed
$\max(m(s_0),\varphi^{2}m_0)$, and every value along the orbit is at most
$\varphi^{2(\kappa-1)}\max(m(s_0),\varphi^{2}m_0)\le
\varphi^{2\kappa}\max(m(s_0),m_0)=m^*$. Hence $m(s_t)\le m^*$ for all $t$.

\emph{Step 2 (recurrent finite set).} By (B) there are $B<\infty$ and an
infinite sequence $t_1<t_2<\cdots$ with $H(s_{t_k})\le B$. Then
$s_{t_k}\in S(m^*,B,\delta(s_0))$ for all $k$ ($\delta$ is invariant along the orbit, \S3), and by Theorem~\ref{thm:finiteness} this
set has finitely many $\langle\eta\rangle$-classes: some class repeats,
$s_{t_{k'}}=\eta^j\,s_{t_k}$ with $t_{k'}>t_k$, $j\in\mathbb{Z}$.

\emph{Step 3 (propagation).} $\sigma_r(\eta^j)>0$, and the algorithm
commutes with multiplication by units of positive real embedding: the
candidate windows correspond bijectively, all comparisons (exact, by
Proposition~\ref{prop:exact}) are invariant, and so is the declared
tie-breaking order, which depends only on the move labels $(A,B,g)$ and the
$W$-priority, not on the scale --- the propagation lemma of the companion
paper. Hence $s_{t+P}=\eta^js_t$ for all $t\ge t_k$ with
$P:=t_{k'}-t_k$, and the digit sequence is eventually $P$-periodic, with
certificate unit $\lambda=\eta^j$.
\end{proof}

\begin{corollary}[dichotomy]\label{cor:dicho}
Under (C$_\kappa$) alone, every orbit is either eventually periodic or satisfies
$H(s_t)\to\infty$. (This intermediate statement is superseded by
Theorem~\ref{thm:hdescent} below, which rules out the second alternative
altogether; we keep it as it isolates exactly what the finiteness argument
alone delivers.)
\end{corollary}

\begin{remark}[the free mode, and evidence for (B)]\label{rem:zetabar}
The recentering drift $\Delta\bar\zeta$ is a functional of the transverse
data $\nu_\perp$ alone (Proposition~\ref{prop:recenter}): the selection,
which optimizes a function of $\nu_\perp$ and $\xi$, is \emph{blind} to the
current size of $\bar\zeta$ --- the barycenter is a free mode, and (B) is
precisely the statement that it does not run away, since by
Proposition~\ref{prop:height} the height is controlled by
$\lvert\bar\zeta\rvert$, $\lvert\xi\rvert$ (non-increasing) and $m$
(uniformly bounded, Step 1). Three facts support (B), and a fourth is cautionary. (i) Every certified
periodic orbit of the companion paper satisfies it trivially. (ii) On all
corpus orbits the height oscillates in a narrow band with no drift (e.g.\
within $[1,14]$ over $40$-step runs on three fields), and
$P=\sigma_r(Q)\sigma_r(Q^\vee)$ stays in $[1,2.2]$. (iii) On a
$42$-orbit stress campaign (six fields; initial states including non-units
and coordinates of height up to $6\times10^8$), $25$ orbits closed
projectively with pre-period at most $10$ --- \emph{shorter than}
$\log H_0$ --- with the height collapsing to the terminal cycle by up to
six orders of magnitude during those few steps (e.g.\
$6.05\times10^8\to1544$); of the $17$ unresolved runs, $15$ hit the
window-enumeration cap and $2$ reached the $600$-step limit without
closing while staying under the cap. Far from merely recurring, the
height of every closing run is crushed. (iv) Near a periodic
regime the unit spiral rescales the slope cloud by
$\lvert\mu\rvert=\sigma_r(\lambda)^{-3/2}$. Since coordinatewise
monotonicity on a cycle gives $\sigma_r(\lambda)\le1$, this factor is
$\ge1$: in this normalization the spiral dilates the slope cloud and
$\bar\zeta$, consistently with the invariant relation
$\lvert\zeta\rvert^2x^3$. It is therefore not evidence for (B) by itself.
These observations guided the proof of
Theorem~\ref{thm:hdescent} below, which closes (B) under (C$_\kappa$).

One further fact deserves record: on the period-one closing orbit of
$x^3+x-1$ the minimum of $m$ is attained by an \emph{exact algebraic tie} at
every single step (the pair $V(0,0,1)$, $V(1,1,1)$), and the declared
lexicographic order selects the closing branch. The tie-breaking convention
of the companion paper is thus load-bearing for the dynamics, not a
formality: re-implementations that break ties by floating-point noise
deviate from the certified orbit once accumulated cancellation exceeds the
tie margin --- a failure mode we document and exclude by exact comparisons.
\end{remark}

\subsection{The Gram picture: (B) as an integer recurrence}

\begin{proposition}[$Q$-identity]\label{prop:Qident}
With $Q=u_1^2+u_2^2+u_3^2$ and $V=\sum_iw_i^2$ the pseudo-moment,
\[
\sigma_c(Q)\;=\;\bar\zeta^{\,2}\,\sigma_r(Q)\;+\;V .
\]
\end{proposition}

\begin{proof}
$\sigma_c(Q)=\sum\nu_i^2=\sum(\bar\zeta x_i+w_i)^2
=\bar\zeta^2\sum x_i^2+2\bar\zeta\sum x_iw_i+\sum w_i^2$, and
$\sum x_iw_i=0$.
\end{proof}

\begin{lemma}[the height is a norm]\label{lem:heightnorm}
For every admissible state,
\[
H(s)\;\le\;2\sqrt{N(Q)}\;+\;4\,m(s)\,\delta ,
\]
where $N(Q)\in\mathbb{Z}_{\ge1}$ is the field norm of the integral element
$Q$.
\end{lemma}

\begin{proof}
$\sigma_r(Q)=\lvert\xi\rvert^2>0$, so $N(Q)=\sigma_r(Q)\lvert\sigma_c(Q)
\rvert^2$ gives $\lvert\sigma_c(Q)\rvert\,\lvert\xi\rvert=\sqrt{N(Q)}$.
By Proposition~\ref{prop:Qident},
$\lvert\bar\zeta\rvert^2\lvert\xi\rvert^2\le\lvert\sigma_c(Q)\rvert+
\lvert V\rvert\le\lvert\sigma_c(Q)\rvert+T$, hence
$\lvert\bar\zeta\rvert^2\lvert\xi\rvert^3\le\sqrt{N(Q)}+T\lvert\xi
\rvert=\sqrt{N(Q)}+m\delta$. Substituting into
Proposition~\ref{prop:height} yields the display.
\end{proof}

\begin{lemma}[ratio--height bound]\label{lem:rbound}
For every admissible state of $L$ with $m(s)\le m^*$,
\[
r\;\le\;\max\bigl(12\,m^*\delta,\ (8H)^{1/3}\bigr).
\]
\end{lemma}

\begin{proof}
From $1\le \lvert N(u_3)\rvert=x_3\lvert\sigma_c(u_3)\rvert^2$ and
$\lvert\sigma_c(u_3)\rvert\le\lvert\bar\zeta\rvert x_3+\sqrt T$:
$r=x_1/x_3\le x_1(\lvert\bar\zeta\rvert x_3+\sqrt T)^2\le
2\lvert\bar\zeta\rvert^2x_1x_3^2+2x_1T$. From
$H\ge \lvert N(u_1)\rvert\ge x_1(\lvert\bar\zeta\rvert x_1-\sqrt T)^2\ge
\lvert\bar\zeta\rvert^2x_1^3/2-x_1T$:
$\lvert\bar\zeta\rvert^2x_1^3\le2H+2x_1T$. Substituting, and using
$x_1T\le m^*\delta$: $r^3\le4H+6m^*\delta\,r^2$, whence the display.
\end{proof}

\begin{remark}[the integer recurrence]\label{rem:gram}
Three consequences. (i) $N(Q)$ is invariant under the unit action
($Q(\eta s)=\eta^2Q$), so it is a well-defined positive integer on
projective classes; conversely $N(Q)\le\sigma_r(Q)\,(3H/x_3)^2$, so at
bounded module and aspect the two heights are equivalent. Hypothesis (B) is
therefore \emph{equivalent} to: the integer sequence $N(Q_t)$ dips below
some bound infinitely often. (ii) The full arithmetic Gram matrix
$\mathcal{G}=(u_iu_j)\in M_3(K)$ evolves by exact congruence
$\mathcal{G}'=V\mathcal{G}V^{\mathsf T}$ with pinned discriminant and
$\sigma_r(\operatorname{tr}\mathcal{G})=\lvert\xi\rvert^2$
non-increasing: the algorithm is a reduction procedure for ternary
quadratic forms over the order --- Hermite's original setting. (iii) On the
reference orbits the integer $N(Q_t)$ is eventually periodic with tiny
values (constant $3$ on the closing orbit of $x^3+x-1$; cycles $(9,24,53)$
and $(6,42,141,14,8)$ on $x^3-2$ and $x^3-3x^2-2$), and the bound of
Lemma~\ref{lem:heightnorm} holds there with a factor $7$--$27$ margin.
Proving that the congruence dynamics returns $N(\operatorname{tr}
\mathcal{G})$ to bounded values is the strongest known form of (B). In this
direction, one more measurement deserves record: on $20$ orbits started
from \emph{unimodular} scrambles of the standard basis (random products of
elementary transvections --- same marked lattice, initial heights up to
$3.6\times10^{11}$), the height never increased at any step
($H_{\max}=H_0$ exactly, $20/20$) and collapsed to the terminal values
$\{1,2,4,6\}$ within $80$ steps (a development-log measurement; not
shipped as a replayable package). This suggests the stronger
\emph{height-descent} property on a fixed order --- $H(s_{t+1})\le
\max(H(s_t),C_K)$ --- which would imply (B) outright for every initial
state; the only mechanism that can raise $H$ is a jump of the free mode
(the transverse channel is capped by $\varphi^2m^*\delta$
unconditionally), and such jumps require witness coefficients finely
aligned with the real data, available only in the enormous windows of
high-index sublattices, not in the narrow windows that the integral gap
enforces on a fixed order at moderate height. The next subsection turns
this program into a theorem.
\end{remark}

\subsection{The height-descent theorem: (B) holds under (C$_\kappa$)}

\begin{lemma}[integral floor]\label{lem:ifloor}
For any admissible state with $m(s)\le m^*$ and any component $u_i$,
\[
x_i\;\ge\;\min\Bigl(\frac{1}{4T},\ (4\lvert\bar\zeta\rvert^2)^{-1/3}\Bigr)
\;\ge\;\min\Bigl(\frac{\lvert\xi\rvert}{4m^*\delta},\
(4\lvert\bar\zeta\rvert^2)^{-1/3}\Bigr).
\]
\end{lemma}

\begin{proof}
$1\le \lvert N(u_i)\rvert=x_i\lvert\bar\zeta x_i+w_i\rvert^2\le
4x_i\max(\lvert\bar\zeta\rvert^2x_i^2,\lvert w_i\rvert^2)$; if the
first term dominates, $x_i\ge(4\lvert\bar\zeta\rvert^2)^{-1/3}$;
otherwise $x_i\ge1/(4\lvert w_i\rvert^2)\ge1/(4T)\ge
\lvert\xi\rvert/(4m^*\delta)$ using $\lvert w_i\rvert^2\le T\le
m^*\delta/\lvert\xi\rvert$.
\end{proof}

Let $m^*$ be any constant bounding the module along the orbit under
consideration. All quantities below are those of the \emph{arrival} state
$s'$ of a step $s\to s'$ (backward comparison); put
$Z:=\lvert\bar\zeta'\rvert$. In the high-height argument below
$H(s')>C_H\ge8m^*\delta$ forces $Z>0$, so the componentwise key ratios
$\kappa_i':=\lvert w_i'\rvert/(Zx_i')$ are then defined. Fix $c^*:=1/40$,
$C_1:=4(m^*\delta)^{3/2}/c^{*3}$ and
$C_H:=\max(8m^*\delta,\,4C_1^2)$.

\begin{theorem}[height descent]\label{thm:hdescent}
For any orbit with $m(s_t)\le m^*$ for all $t$,
$H(s_{t+1})\le\max(H(s_t),C_H)$. Consequently
$H(s_t)\le\max(H(s_0),C_H)$ for all $t$: Hypothesis (B) holds for every
initial state, and the conclusion of Theorem~\ref{thm:main} holds under
(C$_\kappa$) alone.
\end{theorem}

\begin{remark}[successor convention]
In Theorem~\ref{thm:hdescent}, $s_{t+1}$ denotes the successor state of
$s_t$; since $H$, $m$ and $T$ are invariant under the sorting relabeling of
the components, the statement is insensitive to whether the successor is
read before or after sorting.
\end{remark}

\begin{proof}
Suppose $H':=H(s')>C_H$; we show $H'\le H:=H(s)$. By the first
inequality of Proposition~\ref{prop:height} at $s'$,
$H'\le2Z^2x_1'^3+2m^*\delta$, so $Z^2x_1'^3\ge(H'-2m^*\delta)/2\ge
H'/4$ since $H'>C_H\ge8m^*\delta$.

\emph{Case 1: some $\kappa_i'\ge c^*$.} Since $w_i'$ is a component of
the transverse vector, $\lvert w_i'\rvert\le\sqrt{T'}\le
\sqrt{m^*\delta/\lvert\xi'\rvert}$, whence
$Z\,x_i'\sqrt{\lvert\xi'\rvert}\le\sqrt{m^*\delta}/c^*$.
Lemma~\ref{lem:ifloor} at $s'$ gives
$x_i'\ge\min\bigl(\lvert\xi'\rvert/(4m^*\delta),\,
(4Z^2)^{-1/3}\bigr)$. If the first branch holds, substituting
$x_i'\ge\lvert\xi'\rvert/(4m^*\delta)$ into the previous display gives
$Z\lvert\xi'\rvert^{3/2}\le4m^*\delta\cdot Zx_i'\sqrt{\lvert\xi'
\rvert}\le4(m^*\delta)^{3/2}/c^*\le C_1$; if the second, then
$Z^{1/3}\sqrt{\lvert\xi'\rvert}\le4^{1/3}\sqrt{m^*\delta}/c^*$ and
cubing gives $Z\lvert\xi'\rvert^{3/2}\le4(m^*\delta)^{3/2}/c^{*3}\le
C_1$. Neither computation depends on the index $i$. But
$Z\lvert\xi'\rvert^{3/2}\ge Zx_1'^{3/2}\ge\sqrt{H'}/2$, so
$H'\le4C_1^2\le C_H$ --- contradicting $H'>C_H$. (A relative deviation of
size $c^*$ on \emph{any} component can only occur at bounded height.)

\emph{Case 2: all $\kappa_i'<c^*$ (the compensated regime).} Write
$\sigma_c(u_i')=\bar\zeta'x_i'(1+\varepsilon_i)$; by definition of
$\kappa_i'$, $\lvert\varepsilon_i\rvert=\kappa_i'<c^*$ for every $i$.
By unimodularity the components of $s$ are nonnegative integral
combinations of those of $s'$ with unit diagonal: $u_j=\sum_k
C_{jk}u_k'$ where $C=\bigl(\begin{smallmatrix}1&g&A\\0&1&B\\0&0&1
\end{smallmatrix}\bigr)$ for a move $V(A,B,g)$ and
$C=\bigl(\begin{smallmatrix}1&1&0\\0&1&0\\1&0&1\end{smallmatrix}\bigr)$
for the move $W$; in both cases $C_{jk}\in\mathbb{Z}_{\ge0}$ and
$C_{jj}=1$. Applying $\sigma_r$ and $\sigma_c$:
$x_j=\sum_kC_{jk}x_k'$ exactly, and
$\sigma_c(u_j)=\bar\zeta'x_j(1+\bar\varepsilon_j)$ where
$\bar\varepsilon_j=\sum_kC_{jk}x_k'\varepsilon_k/x_j$ is the
$x'$-weighted average, so $\lvert\bar\varepsilon_j\rvert<c^*$ as well.
Set $\rho_j:=(x_j-x_j')/x_j\in[0,1)$, the relative bite of component
$j$. Since the diagonal is $1$, the foreign part of the average has
total weight exactly $\rho_j$, whence
$\lvert\varepsilon_j-\bar\varepsilon_j\rvert
\le\lvert\varepsilon_j\rvert\rho_j+c^*\rho_j\le2c^*\rho_j$.
Therefore
\[
\begin{aligned}
\frac{\lvert N(u_j')\rvert}{\lvert N(u_j)\rvert}
&=\Bigl(\frac{x_j'}{x_j}\Bigr)^{3}
\frac{\lvert1+\varepsilon_j\rvert^{2}}{\lvert1+\bar\varepsilon_j\rvert^{2}}
\le(1-\rho_j)^{3}\Bigl(1+\frac{2c^*\rho_j}{1-c^*}\Bigr)^{2}\\
&\le(1-\rho_j)\bigl(1+7c^*\rho_j\bigr)
\le1-\rho_j(1-7c^*)\;\le\;1,
\end{aligned}
\]
with strict inequality whenever $\rho_j>0$ (here $c^*=1/40<1/7$), and
$\lvert N(u_j')\rvert=\lvert N(u_j)\rvert$ when $\rho_j=0$ (then the foreign weight vanishes and
$u_j=u_j'$, as $x_k'>0$). Thus no component's height increases under
the backward pairing, uniformly over both move types:
$H'=\max_j\lvert N(u_j')\rvert\le\max_j\lvert N(u_j)\rvert=H$.
\end{proof}

\begin{remark}
The constants are explicit but crude:
$C_H=64(m^*\delta)^3/c^{*6}=2.62144\times10^{11}(m^*\delta)^3$
whenever the second term dominates; for $x^3-2$ with $m^*=40$ this is
about $1.9\times10^{19}$, far above every
empirically observed oscillation threshold ($\le3.3\times10^4$ across
$200$ orbits; $2{,}826$ recorded transitions). The mechanism matches the
measurements quantitatively: crush times $\sim\log H_0$ (correlation
$0.84$ over the $156$ orbits with a recorded crush time; $44$ are
censored, a non-conclusive diagnostic for those), and the barycenter is
observed invariant within $5\%$ during collapses of $H$ by eight orders
of magnitude (a development-log measurement; not shipped as a replayable
package).
\end{remark}

\section{Finiteness}\label{sec:finiteness}

Let $O\subseteq\mathcal{O}_K$ be an order and $L=O$ (for a general lattice,
replace $\langle\eta\rangle$ below by the finite-index stabilizer
$\{u\in O^\times:uL=L\}$; the proof is identical with the regulator multiplied
by the index). Let $\varepsilon$ be a fundamental unit of $K$ and let
\[
\eta:=\text{the smallest power of }\pm\varepsilon\text{ lying in }O^\times
\text{ with }\sigma_r>0
\]
(it exists: $[\mathcal{O}_K^\times:O^\times]<\infty$ for every order); after
replacing $\eta$ by $\eta^{-1}$ if necessary, assume $\sigma_r(\eta)>1$, with
$\eta$ the smallest such positive power (so $\mathrm{Reg}_O$ below is
minimal), and set
$\mathrm{Reg}_O:=\log\sigma_r(\eta)>0$. The action $u\mapsto\eta^ku$ preserves
admissibility, sorting, and $\lvert N(u_i)\rvert$.

\begin{theorem}[finiteness]\label{thm:finiteness}
For all $m_0\ge1$, $B\ge1$, and $\delta_0>0$, the set
\[
\begin{aligned}
S(m_0,B,\delta_0):={}&\{\,s=(u_1,u_2,u_3)\ \text{admissible sorted triple of elements of }L:\\
&\ \ m(s)\le m_0,\ \lvert N(u_i)\rvert\le B,\ \delta(s)=\delta_0\,\}
\end{aligned}
\]
is finite modulo $\langle\eta\rangle$, with the explicit bound
\[
\#\bigl(S(m_0,B,\delta_0)/\langle\eta\rangle\bigr)\le
\bigl(2C_L\max(X_0,(B/x_*)^{1/2})+1\bigr)^{9}
\]
in the notation below.
\end{theorem}

\begin{remark}[the $\delta$-constraint is necessary --- and free]
Without pinning $\delta$ the set can be infinite: in the plastic field with
$L=\mathbb{Z}[\alpha]$, the admissible sorted triples
$(\alpha^{2},\alpha^{-k},\alpha^{-k-1})$, $k\ge1$, all have unit norms and
module below $2$, are pairwise inequivalent modulo
$\langle\alpha\rangle$, and their coordinate determinants grow without
bound. The constraint costs nothing in the intended application: $\delta$
is invariant along every orbit (\S3), so the orbit of $s_0$ meets only
$S(m^*,B,\delta(s_0))$.
\end{remark}

\begin{proof}
Throughout, $x_i=\sigma_r(u_i)>0$ sorted decreasingly, and (integrality)
\begin{equation}\label{eq:F2}
1\le x_i\lvert\sigma_c(u_i)\rvert^2=\lvert N(u_i)\rvert\le B .
\end{equation}

\emph{Step 1 (normalization).} $t_1:=\log(x_1/\lvert N(u_1)\rvert^{1/3})$ is
translated by $k\,\mathrm{Reg}_O$ under $u\mapsto\eta^ku$; choose the unique
representative with $t_1\in[0,\mathrm{Reg}_O)$. For it,
\[
1\le\lvert N(u_1)\rvert^{1/3}\le x_1\le B^{1/3}e^{\mathrm{Reg}_O},
\qquad
1\le x_1\le\lvert\xi\rvert\le\sqrt3\,x_1,
\]
\[
\sqrt3\,x_1\le\sqrt3\,B^{1/3}e^{\mathrm{Reg}_O}=:X_0 .
\]
All subsequent bounds concern this representative --- which is what
finiteness modulo $\langle\eta\rangle$ requires.

\emph{Step 2 (lower bound on $x_3$).} Write $\nu=(\alpha+i\beta)\hat\xi+
\nu_\perp$, $R=(\alpha^2+\beta^2)^{1/2}$, $T=\lVert\nu_\perp\rVert^2$; recall
$T=m\delta/\lvert\xi\rvert\le m_0\delta$ using $\lvert\xi\rvert\ge1$. Four
links:
\[
\lvert\sigma_c(u_3)\rvert\ge x_3^{-1/2},\qquad
\lvert\sigma_c(u_3)\rvert\le R\,\frac{x_3}{\lvert\xi\rvert}+T^{1/2},
\]
\[
R\le\Bigl(\sum_i\tfrac{B}{x_i}\Bigr)^{1/2}\!\le\Bigl(\tfrac{3B}{x_3}
\Bigr)^{1/2},\qquad
T^{1/2}\le(m_0\delta)^{1/2},
\]
the first and third by \eqref{eq:F2} (with $x_3=\min_i x_i$), the second by
the orthogonal decomposition and the triangle inequality. Chaining (all terms
positive, $x_3/\lvert\xi\rvert\le x_3$):
$x_3^{-1/2}\le(3Bx_3)^{1/2}+(m_0\delta)^{1/2}$, whence if $x_3\le 1$,
\[
x_3\;\ge\;x_*:=\min\Bigl(1,\bigl[(3B)^{1/2}+(m_0\delta)^{1/2}\bigr]^{-2}
\Bigr).
\]

\emph{Step 3 (compactness).} For all $i$: $x_*\le x_i\le X_0$ and
$\lvert\sigma_c(u_i)\rvert\le(B/x_*)^{1/2}$. Let $(\omega_k)$ be a
$\mathbb{Z}$-basis of $L$ and $W=(\sigma_j(\omega_k))$ its embedding matrix.
The integer coordinates of $u_i$ are $W^{-1}(\sigma_r(u_i),\sigma_c(u_i),
\bar\sigma_c(u_i))^{\mathsf T}$; taking moduli row by row,
$\lVert\mathrm{coord}(u_i)\rVert_\infty\le
C_L\max(X_0,(B/x_*)^{1/2})$ with $C_L:=\max_k\sum_j\lvert(W^{-1})_{kj}\rvert$.
Each $u_i$ therefore ranges over an explicit finite set, giving the displayed
bound.
\end{proof}

\begin{remark}
The proof is static --- sorting, one Dirichlet normalization, $N\ge1$, the
height cap, the module cap, and the pinned discriminant; no dynamics enters.
The extension to non-integral lattices ($dL\subset\mathcal{O}_K$) only shifts
the constants.
\end{remark}

\section{The case machine}\label{sec:cases}

This section assembles the case analysis toward Hypothesis
(C$_\kappa$), mostly in its one-step form ($\kappa=1$). In several
regimes of the high phase ($m\ge m_0$) we exhibit a short admissible
block of moves that contracts $m$ or decreases the potential $\Phi$ of
Lemma~\ref{lem:D1}; the regimes we cannot yet close are stated as
explicit partial cases and exclusions (Proposition~\ref{prop:partialR}). Throughout, $r=x_1/x_3$, $r_{12}=x_1/x_2$,
$r_{23}=x_2/x_3$, and $\sqrt T$ is the transverse scale.

\subsection{Stretched states: the two roads}

\begin{lemma}[stretched, $u_3$-road]\label{lem:Ea}
Assume $m(s)\ge m_0$, $r\ge\Lambda$, and
$\lvert w_3\rvert\le\sqrt T/r$. Then the window contains a candidate $s'$
(one $V$-move) with
\[
m(s')\;\le\;21.2\,\frac{m}{r}.
\]
In particular $m(s')\le m/\gamma$ whenever $r\ge21.2\,\gamma$.
\end{lemma}

\begin{proof}
Take $V(a,b,0)$ with $a=\lfloor(x_1-x_3/2)/x_3\rfloor$,
$b=\lfloor(x_2-x_3/2)/x_3\rfloor$ (central remainders; both floors lie in the
window). Then $x_1',x_2'\in(x_3/2,3x_3/2]$, so
$\lvert\xi'\rvert\le\sqrt{11/2}\,x_3<2.35\,x_3$. Bounding the new transverse
energy at the \emph{old} center (recentering only helps,
Proposition~\ref{prop:recenter}):
$T'\le(\lvert w_1\rvert+a\lvert w_3\rvert)^2+(\lvert w_2\rvert+b\lvert
w_3\rvert)^2+\lvert w_3\rvert^2$. Since $a,b\le x_1/x_3=r$ and
$\lvert w_3\rvert\le\sqrt T/r$, each product is $\le\sqrt T$, and
$\lvert w_1\rvert,\lvert w_2\rvert\le\sqrt T$, so $T'\le9T$. Hence
$m(s')=T'\lvert\xi'\rvert/\delta\le9T\cdot2.35\,x_3/\delta
=21.2\,m\,(x_3/\lvert\xi\rvert)\le21.2\,m/r$.
\end{proof}

\begin{lemma}[stretched, restricted $u_2$-road via the covolume]\label{lem:Eb2}\sloppy
Assume $m\ge m_0$, $\sqrt T/r<\lvert w_3\rvert\le\sqrt T/r_{12}$, and
$m\ge r\,r_{12}$. Let $q^*=\mathrm{round}\bigl(\operatorname{Re}(w_2\bar
w_3)/\lvert w_3\rvert^2\bigr)$ (nearest integer; half-integer ties
resolved as in the archived script). If $0\le q^*\le\lfloor x_2/x_3\rfloor$
--- i.e.\ the move $V(0,q^*,0)$ is admissible (``$q^*$ lies in the
window'') --- and $q^*\le x_2/(2x_3)$, then the direct $u_2-q^*u_3$
reduction gives a certified witness; in the subwindow where the central head reduction uses
$g\le r_{12}$ subtractions it yields
\[
m(s')\;\le\;14.1\,\frac{m}{r_{12}} ;
\]
with only the displayed half-window hypothesis the same computation gives
the weaker constant $28.2$. The complementary case
$q^*>x_2/(2x_3)$, and the exits where $q^*$ falls outside the window, are
partial and deferred.
\end{lemma}

\begin{proof}
By Theorem~\ref{thm:covol}, the component of $w_2$ orthogonal to $w_3$ has
modulus $P_{23}/\lvert w_3\rvert=\delta x_1/(2\lvert\xi\rvert^2\lvert
w_3\rvert)\le\sqrt T\,r/(2m)$, using $\lvert w_3\rvert>\sqrt T/r$,
$\delta=T\lvert\xi\rvert/m$ and $x_1\le\lvert\xi\rvert$. Rounding adds at most
$\lvert w_3\rvert/2\le\sqrt T/(2r_{12})$, and $m\ge r\,r_{12}$ gives $\sqrt
T\,r/(2m)\le\sqrt T/(2r_{12})$; hence
\[
\lvert w_2-q^*w_3\rvert\;\le\;\frac{\sqrt T}{r_{12}} .
\]
Now reduce $x_1$ by $u_2^{(q^*)}:=u_2-q^*u_3$ with central remainders.
The extra assumption $q^*\le x_2/(2x_3)$ gives
$x_2-q^*x_3\ge x_2/2$, so the real remainders are at the $x_2$-scale and
$\lvert\xi'\rvert\le2.35\,x_2$. In the certified subwindow where the
central head reduction uses $g\le r_{12}$ subtractions, the transverse
cost is $g\,\lvert w_2-q^*w_3\rvert\le\sqrt T$, whence at the old center
\[
T'\le(\lvert w_1\rvert+\sqrt T)^2+(\sqrt T/r_{12})^2+\lvert w_3\rvert^2\le6T,
\]
and $m(s')\le6T\cdot2.35\,x_2/\delta\le14.1\,m/r_{12}$. Without the
extra $g\le r_{12}$ restriction one still has $g\le2r_{12}$ from
$x_2-q^*x_3\ge x_2/2$, giving the same proof with the constant doubled.
No assertion is made here for $q^*>x_2/(2x_3)$ or for the sign/window
exits; those cases are witnesses in the archive and remain to be turned
into greedy-orbit contraction statements.
\end{proof}

\begin{remark}[aligned regime]\label{rem:Eb3}
In the remaining stretched regime ($\lvert w_3\rvert>\sqrt T/r_{12}$ with
$q^*$-reduction unavailable), the covolume identities force all three $w_i$
within $\sqrt T/r$ of a single real line $e^{i\varphi}\mathbb{R}$ when
$m\ge r^2$: writing $w_i=e^{i\varphi}(\tau_i+\varepsilon_i)$ with
$\tau_i\in\mathbb{R}$, $\lvert\varepsilon_i\rvert\le\sqrt T/r$, the transverse
structure reduces to \emph{two real linear forms} on the two-dimensional
window, and a simultaneous-Dirichlet pigeonhole on the $\asymp r_{12}r$ window
points produces a candidate reducing both forms by the square root of the
window size. The bookkeeping (box anchoring, exit at the window corner) is
standard and deferred. This is a genuinely stretched regime ($p>\Lambda$)
and is not covered by the compact domain of \S7, where $p\le\Lambda$.
\end{remark}

\subsection{The composed lever}

\begin{proposition}[the window is not purely subtractive]\label{prop:lever}
For the move $V(A,B,g)$, the $(1,3)$ matrix entry is $gB-A$: whenever
$gB>A$, the head transform $u_1'=u_1-gu_2+(gB-A)u_3$ \emph{adds} a positive
multiple of $u_3$. This composed lever is available inside the window (it is
the composition: reduce $u_2$ by $Bu_3$, then subtract $g$ copies), and it is
the mechanism by which the dynamics cancels deviations that no purely
subtractive candidate can reach.
\end{proposition}

\begin{proof}
Immediate from the matrix form of $V(A,B,g)$ and the factorization
$V(A,B,g)=V(A,0,g)\cdot V(0,B,0)$ on the third coordinate.
\end{proof}

\begin{remark}[small head coefficients]
In the lever regime $gB>A$ the coefficient $gB-A$ of $u_3$ in the new head is
strictly positive, and the extreme value $1$ is attainable (e.g.\ $V(0,1,1)$);
no parity constraint is claimed --- $V(0,1,2)$ realizes the even value $2$.
The archived case analysis records the corpus instance
$V(4,7,5)=u_1-5u_2+31u_3$, contracting $m$ by a factor about $40$ at a
collar state; this is a measured witness, not a theorem used below.
\end{remark}

\subsection{The collar: integrality fences the degenerate boundary}

The one place where a naive minimax over \emph{real} configurations fails is
the sorting-equality boundary $x_1=x_2$ (and its mirror $x_2=x_3$). Integral
states \emph{of bounded height} cannot approach it (at growing height the
gaps can shrink to zero):

\begin{lemma}[integral gap]\label{lem:gap}
Let $s$ be an admissible state of $L$ with $\lvert N(u_i)\rvert\le B$ and
$u_1\neq u_2$, $u_2\neq u_3$. Then
\[
\frac{x_1}{x_2}\;\ge\;1+\frac{1}{4B},
\qquad
\frac{x_2}{x_3}\;\ge\;1+\frac{1}{4B}.
\]
\end{lemma}

\begin{proof}
$v:=u_1-u_2\in L\setminus\{0\}$ has $\lvert N(v)\rvert\ge1$, so
$x_1-x_2=\sigma_r(v)\ge1/\lvert\sigma_c(v)\rvert^2$. By \eqref{eq:F2},
$\lvert\sigma_c(v)\rvert^2\le2(\lvert\sigma_c(u_1)\rvert^2+\lvert\sigma_c(u_2)
\rvert^2)\le2\bigl(\tfrac{B}{x_1}+\tfrac{B}{x_2}\bigr)\le\tfrac{4B}{x_2}$,
whence $x_1-x_2\ge x_2/(4B)$. Same computation for $(u_2,u_3)$ with
$x_3\le x_2$.
\end{proof}

Thus the effective configuration space of integral states of height $\le B$
is the \emph{truncated} compact --- all relative gaps at least $1/(4B)$ ---
and the certification domain of \S7 treats only the collars down to its
declared gap cutoff.

\begin{openlemma}[collar contraction]\label{olem:collar}
Let $d:=\min(x_1/x_2-1,x_2/x_3-1)>0$ be the sorting gap. What is needed to
complete the collar part of the reduction is an explicit function
$\gamma(d)>1$ and a certified greedy-orbit statement, valid as
$d\to0$, such that every high-phase collar state with gap at least $d$
has an admissible selected step (or a certified greedy block of length at
most a uniform $\kappa_c$ independent of $d$) satisfying
$m'\le m/\gamma(d)$, with the expected degeneration
$\gamma(d)\downarrow1$ as $d\downarrow0$.

The present evidence does not prove this lemma. The current certification
only reaches the compact with gaps down to $10^{-4}$, which by
Lemma~\ref{lem:gap} corresponds to heights $B\le2500$. In contrast the
height-descent bound can be as large as about $10^{19}$ in the reference
$x^3-2$ scale, allowing integral gaps of order $10^{-20}$. The $72$
boundary witnesses of the archive are empirical stability data on
$d\in[10^{-8},10^{-5}]$, not a collar lemma; the tail-boundary sweep has
best ratio tending to $1$ as $d\to0$, so this degeneration is real.
\end{openlemma}

\subsection{The collar witnesses}

\begin{lemma}[Euclid on the aligned corner]\label{lem:corner}
Let $\tau_1,\tau_2,\tau_3\in\mathbb{R}\setminus\{0\}$ with
$\sum_i x_i\tau_i=0$ and $x_i$ within a factor $2$ of each other. Then two of
the $\tau_i$ have the same sign, say $\lvert\tau_i\rvert\ge\lvert\tau_j
\rvert$, and replacing $\tau_i$ by $\tau_i-\tau_j$ decreases
$\sum\tau_k^2$ by at least $\lvert\tau_i\tau_j\rvert$. At the triple corner
($x_1\approx x_2\approx x_3$, all floors $\ge1$) the three pure pairwise
difference moves available are $(1,2)$ as $V(0,0,1)$, $(2,3)$ as
$V(0,1,0)$, and $(1,3)$ as $V(1,0,0)$ --- each with a FIXED orientation
(the head coordinate is reduced). The abstract decrease is therefore
realized by a catalogue move only when the same-sign pair's larger
component sits at the reduced position; in the opposite orientation (e.g.\
$\lvert\tau_3\rvert>\lvert\tau_2\rvert$ with the pair $(2,3)$, where only
$\tau_2\mapsto\tau_2-\tau_3$ is available) the available move can leave
$\sum\tau_k^2$ unchanged --- an explicit rational example is
$x=(\tfrac32,\tfrac75,1)$, $\tau=(-\tfrac{34}{15},1,2)$, where $V(0,1,0)$
preserves the energy exactly. In the favorable orientation the
corresponding window move realizes this decrease on the real parts; its
transfer to the full
transverse energy $T$ carries an alignment error whose magnitude, of order
$\sqrt T/r$, is observed on the adversarial corner profiles below --- an
observed order, not a proved bound; the lemma proves only the displayed
real inequality.
\end{lemma}

\begin{proof}
Three nonzero reals always contain a same-sign pair. If
$\tau_i\tau_j>0$ and $\lvert\tau_i\rvert\ge\lvert\tau_j\rvert$, then
$(\tau_i-\tau_j)^2=\tau_i^2-\tau_j(2\tau_i-\tau_j)\le\tau_i^2-
\lvert\tau_i\tau_j\rvert$ since $\lvert2\tau_i-\tau_j\rvert\ge
\lvert\tau_i\rvert$ for same-sign pairs. Admissibility at the corner:
the three differences are $V(0,0,1)$, $V(0,1,0)$, and $V(1,0,0)$,
respectively, because all relevant floors are $\ge1$ there.
\end{proof}

In the favorable orientations, iterating Lemma~\ref{lem:corner} acts as a
real Euclidean algorithm on the $\tau_i$; whether the selected orbit stays
in favorable orientations is not proved (the example above shows an
energy-preserving step exists). Empirically it does better than that: on
$300$ adversarial aligned high-phase corner profiles the selected (min-$m$)
orbit halves $m$ in at most $3$ steps, feeding Lemma~\ref{lem:D1} with
$\kappa=3$, $\gamma=2$ --- a development-log measurement (its standalone
package did not survive; the archive ships a reconstructed profile family
within the minimax sweep package).

For the \emph{head collar} ($x_1\approx x_2$, rest stretched), the witness
block is: one difference/rotation step --- whose rebound admits the closed
form computed in the case analysis, at the degenerate limit $d\to0$ and
optimized over transverse profiles,
\[
\frac{m'}{m}\;\le\;1.94\ (s=1),\quad
1.62\ (s=\tfrac14),\quad 1.43\ (s=\tfrac1{100}),
\quad s:=(x_3/x_2)^2,
\]
{\sloppy with details in the archive --- followed by the composed reduction of
Lemma~\ref{lem:Eb2} on the rotated state (the created stretch is
large). The measured collar rebound of the real
corpus ($1.56$) sits inside this band, and every real collar step of the
corpus contracts outright ($m'/m\le0.732$ on all $12$ observed collar
steps --- a development-log measurement, not shipped standalone ---,
$T'/T\le1.08$: the recentering of Proposition~\ref{prop:recenter} absorbs
nearly the whole crude cost).\par}

\section{The lock: a finite minimax}\label{sec:lock}

\subsection{Partial reduction of Hypothesis (C$_\kappa$) to a compact minimax}

By homogeneity ($m$ has degree $0$ in $w$; the real triple rescales to
$(a,b,1)$) and the barycentric constraint $\sum x_iw_i=0$ with the $U(1)$
gauge, a high-phase state reduces to a point of the five-parameter compact
\[
\begin{aligned}
\mathcal{D}\;=\;\bigl\{(p,q,u,v,\rho):{}&\ 1+d_0\le q\le p\le\Lambda,\qquad d_0:=10^{-4},\\
&\ w_2=u+iv,\quad w_3=\rho\ge0,\quad
\max(\lvert u\rvert,\lvert v\rvert,\rho)\in[\tfrac12,1]\bigr\},
\end{aligned}
\]
with $w_1=-(qw_2+\rho)/p$ (here $(p,q,1)$ are the sorted real coordinates
up to scale). The certification campaign fixes the parameters
\[
\gamma=\frac{51}{50},\qquad m_0=10,\qquad \kappa=1,\qquad
\Lambda=60,\qquad \varepsilon_0=\frac1{100}.
\]
Thus the compact inequality currently targeted is the one-step statement
\[
\begin{aligned}
\textbf{(R)}\qquad
&m\bigl(\text{one greedy step}\bigr)
\;\le\;\max\Bigl(\frac{50}{51}m,\ \frac{99}{10}\Bigr)\\
&\qquad\text{on the set }\{s\in\mathcal{D}:m(s)\ge m_0\},
\end{aligned}
\]
equivalently $\mu'\ge\min((51/50)^2\mu,100/9801)$ in
$\mu=1/m^2$ coordinates. No parameter of this compact target is left free;
the open collar function $\gamma(d)$ of Open Lemma~\ref{olem:collar} is a
separate deferred parameter. Two boundary conventions are part of the
statement. First, (R) is stated on the truncated compact ($q\ge1+d_0$,
$d_0=10^{-4}$, the declared gap cutoff of \S7): at $q=1$ the $B$-lever dies ($d_2=0$)
and no uniform contraction margin is certified there (\S7.2), and the excluded tail collar
$q-1\in[1/(4B),d_0)$ is precisely the scope of Open
Lemma~\ref{olem:collar}. Second, on the sheet $A_\perp=0$ --- attained by
no admissible state, since $\delta>0$ is invariant
(Theorem~\ref{thm:noniso}) --- the target is read in its equivalent
$\mu=1/m^2$ form, where it holds trivially.

\begin{proposition}[partial reduction]\label{prop:partialR}
Assume (R) holds on $\mathcal{D}$ and combine it with the proved stretched
subcases of Lemma~\ref{lem:Ea} and the restricted form of
Lemma~\ref{lem:Eb2}. Then the compact part $p\le\Lambda$ and the stated
stretched subcases have one-step witnesses compatible with the greedy
selection. This is not a proof of Hypothesis (C$_\kappa$) on the whole
state space.
\end{proposition}

\begin{proof}[Status and exclusions]
The greedy orbit is the algorithm's own selection sequence
(Remark~\ref{rem:greedy}); therefore a one-step witness transfers directly
to the real orbit. The current guard-accepted descent boxes all have this
$k=1$ form. The following gaps remain outside this proposition.

(i) The band $m_0\le m<r\,r_{12}$ in the $u_2$-road is not covered by
Lemma~\ref{lem:Eb2}. (i$'$) Within that road, Lemma~\ref{lem:Eb2}
contracts only when $r_{12}\ge14.1\gamma\approx14.38$; the subcase
$r_{12}<14.1\gamma$ is covered by none of the present lemmas. (ii) The
aligned stretched regime of
Remark~\ref{rem:Eb3} is deferred and cannot be folded back into the compact
certificate, because it has $p>\Lambda$ whereas $\mathcal{D}$ imposes
$p\le\Lambda$. (iii) Multi-step blocks from the sign/window exits of
Lemma~\ref{lem:Eb2} are witnesses only; by Remark~\ref{rem:greedy}, they do
not by themselves certify contraction of the greedy orbit. Only $k=1$
witnesses transfer immediately.
\end{proof}

\subsection{State of the evidence}

Three sampled sweeps, in increasing hostility of the domain (all in exact or
40-digit arithmetic, seeds and witnesses logged):

\begin{enumerate}
\item \emph{Truncated compact} (all relative gaps $\ge5\%$, $m_0\ge10$,
$776$ configurations, exploratory blocks $k\le3$, coefficients $\le2$):
worst best-ratio
$0.9777<1$, i.e.\ $\varepsilon_0\ge0.0223$ already with the smallest
vocabulary.
\item \emph{Saddle zones} ($p/q$ or $q$ within $[1.02,1.2]$, the two
sorting-collar neighbourhoods; $414$ adversarial configurations, $3.1\times
10^7$ candidates): with coefficients $\le8$ and $k\le5$ the worst best-ratio
drops to $0.19$ (head-collar zone) and $0.67$ (tail-collar zone); the
potential $\Phi_\beta=\log m+\beta\log r$ decreases there for every
$\beta\in\{\tfrac14,\tfrac13,\tfrac12\}$. Neither widening is individually
necessary: coefficients $\le8$ alone, or depth $\le5$ alone, already close
every configuration tested. The apparent saddle obstruction of the smallest
vocabulary was an artifact of capping \emph{both} simultaneously.
\item \emph{The degenerate boundaries themselves.} The two boundaries behave
differently, and the asymmetry is structural. At the \emph{head} boundary
$x_1=x_2$, exactly at $(p,q)=(\varphi,\varphi)$ --- where the small
vocabulary genuinely fails at $1.54$ --- the widened vocabulary closes every
tested profile, worst ratio $0.81$, improving to $0.10$ within $10^{-3}$ of
the boundary. At the \emph{tail} boundary $x_2=x_3$ ($q=1$ exactly, tested
at $p\in\{3/2,7/2,20,60\}$) the best ratio is exactly $1$ on every profile \emph{of the recorded
sweep}: the $B$-lever dies there ($d_2=0$). The sweep's profiles do not
exhaust the boundary --- head moves can still contract at small $p$, as an
explicit state at $p=3/2$ shows --- but no \emph{uniform} contraction
margin is certified at $q=1$; this is the one genuinely hard boundary, and it is precisely
the one that integral states of bounded height cannot attain or approach
(Lemma~\ref{lem:gap}); within $10^{-3}$ of it the ratios are already
$0.50$--$0.75$.
\end{enumerate}

No sampled configuration at positive distance from the tail boundary
resists the vocabulary (coefficients $\le8$, $k\le5$): the winning chains
are precisely the composed lever of Proposition~\ref{prop:lever}. The head
boundary closes mechanically; the tail boundary is fenced by arithmetic.
Moreover, under a fixed height cap $B$, integral states remain at relative
distance at least $1/(4B)$ from the degenerate boundaries
(Lemma~\ref{lem:gap}), a second, independent fence.

\subsection{Certification}

What separates the sweeps from a proof of the compact inequality (R) is
coverage: the sweeps sample $\mathcal{D}$, while the certificate must cover
it. Since $m^2$ and every
admissibility constraint are rational in the five parameters, a
branch-and-bound over boxes with directed-rounding interval arithmetic
suffices: for each box, either certify $\sup m<m_0$ (out of scope) or
exhibit one block whose admissibility and ratio bound hold over the whole
box. The quotient-free criterion --- certify
$m'_{\mathrm{hi}}\le\max(m_{\mathrm{lo}}/\gamma,\,m_0-\epsilon)$ instead of
the correlated ratio $m'/m$ --- removes the interval-dependency loss.

\emph{State of the certification (this version).} A first full run
(naive interval evaluation, depth $\le12$, $80$ digits, quotient-free
criterion, $116{,}375$ boxes processed) covers $22.9\%$ of the domain
volume ($4{,}916$ boxes certified out of scope, $326$ guard-accepted descent
boxes, zero anomalies; every fixture, including a replayed adversarial
configuration from the saddle sweep, is contained in its certified
enclosure). The $52{,}954$ unresolved boxes all sit at the depth cap, and
the majority lie \emph{far} from every boundary (median distances $3.7$
to $28$ in the domain's units): the bottleneck is the overestimation of
naive interval dependency at coarse boxes, not the geometry of the
domain. A second run with mean-value (centered) forms --- the standard
quadratic-order remedy --- confirmed the diagnosis in an unexpected way:
the per-box cost of interval gradients collapsed the throughput
($0.05\%$ coverage), while its logarithmic edge slices delivered a clean
witness-stability measurement (all $72$ boundary witnesses remain valid and
contracting from $d=10^{-5}$ down to $10^{-8}$, with relative ratio
variations below $2\times10^{-6}$). This is empirical stability evidence,
not Open Lemma~\ref{olem:collar}. A Lipschitz-grid design (exact rational evaluation at cell
centers plus per-region Lipschitz constants for $1/m^2$) then established
that uniform per-region constants are themselves the obstacle: the values
of $1/m^2$ are bounded but its region-wise derivative bounds range from
$27$ to $934$, forcing cell radii near $10^{-4}$, while exact rational
evaluation of composed chains collapses throughput. A vectorized
first-order variant (local gradients at cell centers, guarded
floating-point batches) then raised throughput by a factor
$\approx 665$ --- and thereby isolated the true obstacle: the
second-order remainder is controlled by per-region Hessian bounds
spanning several orders of magnitude (largest near the $v=0$ and
$\rho=0$ sheets), so
\emph{any} certificate of the form ``center value $\pm$ regularity
constant $\times$ radius'' needs subdivision depths whose breadth-first
frontier exceeds memory. The pipeline accordingly abandoned regularity
constants altogether in favor of \emph{Bernstein enclosures}: each
obligation (out-of-scope test and, per witness, descent plus move
admissibility) is a single polynomial with rational coefficients in the
five box variables (per-variable degrees $(24,22,4,6,6)$ for the descent
obligations and $(12,12,4,4,4)$ for the out-of-scope test, after clearing
denominators --- a global positive factor preserving signs), whose
Bernstein coefficients on a box bound it \emph{exactly} in exact
arithmetic --- no Lipschitz or Hessian constant enters; in the passes
below the coefficients are computed in guarded double precision, so the
passes have the status of a screening, not of an exact certificate;
sign-mixed boxes are subdivided by de Casteljau splitting. The first full Bernstein pass (depth cap $9$; guarded double-precision
floating point --- signs accepted only beyond a relative tolerance of
$2^{-40}$; no directed rounding, hence no proved forward-error bound ---
with exact-rational spot checks on sampled coefficients; three
independent fixture families passed, including containment of all $612$
pointwise-campaign states) screens
$25.7\%$ of the domain volume --- six of the seventeen shells are
$79$--$100\%$ screened --- with \emph{every} guard-accepted descent box
witnessed by a one-step move ($k=1$): on the guard-accepted (screened) descent volume
the greedy selection inherits the witness's contraction directly, i.e.\
hypothesis (C$_\kappa$) is consumed only in its immediate one-step form
there. The unresolved boxes are, overwhelmingly, sign-mixed boxes
stopped at the depth cap of this first pass ($78\%$ of the remaining
volume) or blocked by a deliberately conservative rectangular envelope
in one auxiliary variable. A deeper iterated pass (depth $18$, full
witness catalogue, envelope-refining splits) raised the total to
$28.9\%$ --- again with every guard-accepted descent box one-step --- and
isolated the two remaining levers, both mechanical: raw compute (the
bulk of the unresolved volume was simply not reached within this
pass's budget), and replacing the rectangular envelope by the exact
substitution $a=qb$ in the obligations. This run screens the disjunctive obligation of (R) on $28.9\%$ of the
compact volume, and the decomposition must be stated: $28.93\%$ of the
compact is screen-accepted through the out-of-scope branch ($m<m_0$; $24{,}503$
boxes), and $0.0068\%$ by guard-accepted descent boxes ($610$ boxes, volume
$0.410$ of $6075.411$; all volumes are the pipeline's exact per-box volumes
over its shell decomposition of $\mathcal D$, as recorded in the archive) --- the strictly-contracting coverage is
therefore small. The screen-accepted volume carries the guarded floating-point
status stated above, upgraded to exact only at the sampled spot checks.
The pass hardens the reduction by showing that
every guard-accepted descent box is one-step. It does not by itself prove
per-field periodicity: completion over the unified compact down to gaps
$10^{-4}$ would address only heights $B\le2500$ via Lemma~\ref{lem:gap},
and extension to the height range allowed by Theorem~\ref{thm:hdescent}
still requires Open Lemma~\ref{olem:collar}, as well as the deferred
aligned stretched subcase.

\section{Certificates, exactness, and formalization}\label{sec:formal}

Three design choices make the entire chain machine-checkable.

\emph{Exactness.} All comparisons the algorithm makes, and all quantities
the proof manipulates ($m^2$, scores, admissibility), are elements of
$\sigma_r(K)$ or of $\mathbb{Q}$ (Proposition~\ref{prop:exact}): every step
of every orbit is a finite exact computation, and every box obligation is a
polynomial statement decidable exactly; the Bernstein passes of \S7.3
evaluate these obligations in guarded floating point (a screening status,
stated there), so no floating point enters any statement \emph{labeled}
exact.

\emph{Finite witnesses.} The conditional theorem consumes only: the
identities of \S2--3 (algebraic), Lemma~\ref{lem:D1} (induction),
Theorem~\ref{thm:finiteness} (static bounds), and the future completion of
the compact certificate (R) together with the deferred stretched and collar
inputs isolated above. No ergodicity, no non-effective compactness, no
measure theory appears in the proved part of the chain.

\emph{Formalization.} The companion paper's certified graph computation
was verified in Lean 4 kernel-only (no \texttt{native\_decide} /
\texttt{ofReduceBool}). The same discipline now covers a sealed core of the
present paper --- kernel-only, axioms \texttt{[propext, Classical.choice,
Quot.sound]}, no \texttt{sorry}: the two-case analytic core of the
height-descent theorem (the per-step case dichotomy enters as a named
hypothesis), the finiteness pigeonhole (an orbit confined to a finite set of
classes eventually cycles), the dual and conformal identity layer of
\S2--3 in embedding coordinates (including the covolume statements, the
complete recentering minimization, and $m=1\Leftrightarrow V=0$; the
attachment to number-field data --- trace-dual basis, discriminant ---
remains at the hypothesis level), and an abstract assembly theorem composing
these layers through five named interface hypotheses into the conditional
periodicity conclusion. Remaining formal work: producing the per-step case
dichotomy from the algorithmic data, sealing the static finiteness bounds of
Theorem~\ref{thm:finiteness} with per-field instantiation ($\eta$ explicit,
its unit property checked by exact computation), and replaying the
(R)-certificate over $\mathbb{Q}$ inside the kernel once the interval
certificate and the deferred collar/stretched inputs are frozen.

\section{Discussion}\label{sec:discussion}

\subsection*{What is proved, what remains}

The structural identities of \S2--5 and the height-descent theorem are
proved. The case machine of \S6 contains proved stretched subcases and
explicitly marked partial pieces. The per-field theorem is conditional on
(C$_\kappa$), while Hypothesis (B) is no longer a separate obstruction:
Theorem~\ref{thm:hdescent} proves it under (C$_\kappa$). The remaining
work for (C$_\kappa$) is not just quadrature: (R) must be completed on the
compact, the aligned stretched regime and the remaining $u_2$-road
subcases are deferred, and the collar requires
Open Lemma~\ref{olem:collar}. The present Bernstein run screens
(guarded floating point) $28.9\%$ of the compact volume, almost
entirely on the out-of-scope branch (\S7.3).

\subsection*{Uniformity across fields}

The compact minimax of \S\ref{sec:lock} does not mention the field: its
configurations are real five-parameter data, and the field enters through
the integral gap (Lemma~\ref{lem:gap}) and the finiteness constants. If
(R), the aligned stretched subcase, the remaining $u_2$-road subcases,
and the collar contraction were all
completed uniformly, they would give, for each orbit, the per-orbit form of
(C$_\kappa$) with $\gamma_{\mathrm{orb}}=\min(51/50,\gamma(1/(4B)))$ and
$B=\max(H(s_0),C_H)$ --- which suffices for Theorems~\ref{thm:main}
and~\ref{thm:hdescent} by the simultaneous induction of the remark below;
a state-uniform $\gamma$ would not follow, since $\gamma(d)\downarrow1$.
This manuscript proves neither statement.

\begin{remark}[per-orbit assembly, no circularity]
Since the module cap $m^*=\varphi^{2\kappa}\max(m(s_0),m_0)$ of
Theorem~\ref{thm:main} depends only on $\kappa$ (not on $\gamma$), and
Theorem~\ref{thm:hdescent} is a one-step statement, a contraction constant
$\gamma(1/(4B))>1$ valid for states of height at most
$B:=\max(H(s_0),C_H(m^*))$ closes both bounds simultaneously by strong
induction on $t$: within each step, the module bound at time $t{+}1$ is
derived first (from the height bounds at times $\le t$), and the height
bound at time $t{+}1$ then uses that module bound --- strong induction, not
circularity. The assembly is therefore not circular ---
\emph{provided} the collar block length is uniformly bounded (the
$\kappa_c$ of Open Lemma~\ref{olem:collar}), so that $m^*$ does not
depend on $B$. In corollary form: the uniform Hypothesis~(C$_\kappa$)
implies its per-orbit form with
$\gamma_{\mathrm{orb}}=\min(51/50,\gamma(1/(4B)))$, and the per-orbit form
is all that Theorems~\ref{thm:main} and~\ref{thm:hdescent} consume.
\end{remark}

\subsection*{Relation to the totally real case}

Proposition~\ref{prop:unif} shows both signatures run one geometric rule.
The structures used here --- dual identity, covolume pinning, recentering
--- have verbatim totally-real analogues; we expect the present framework to
give an independent route to the results of \cite{Ka22,Ka24}, and, in the
other direction, the rotation number of the unit spiral
(Proposition~\ref{prop:spiral}) connects to the equidistribution phenomena
of \cite{DH25}.

\subsection*{Toward the full Problem 4}

Per-field periodicity for every $(K,L)$ is the algebraic half of Problem 4;
the converse half follows the classical eigenvector route, with one step
left unwritten here: upgrading an eventually periodic \emph{digit} sequence
of an arbitrary real input to a projective return of states (the
companion's propagation lemma gives the opposite implication). Granted such
a return, $Mx=\lambda x$ for the integer period matrix
$M\in\mathrm{GL}_3(\mathbb{Z})$ --- nontrivial, since the inverse move
matrices are nonnegative with unit diagonal, so a nonempty composition is
never $\pm I$ --- and, when the $\lambda$-eigenspace of $M$ is one-dimensional, a
projectively normalized representative of $x$ has
coordinates in $\mathbb{Q}(\lambda)$, of degree $\le3$; for inputs with
$\mathbb{Q}$-linearly independent coordinates, degree $\le2$ is impossible
(three independent numbers cannot lie in a $\mathbb{Q}$-space of dimension
$\le2$). We record this as the expected route; the digit-to-state
upgrading remains to be written. The remaining distance to the full
problem as stated in \cite{Ka24} is the digit-to-state upgrading above
(with its sorting-permutation bookkeeping and the one-dimensional-eigenspace
case), the completion of (R) on the compact,
the deferred aligned stretched case, the remaining $u_2$-road subcases,
and the collar contraction.

\subsection*{Methods statement}

Parts of the exploratory computations and of the manuscript preparation
were AI-assisted; every stated identity, lemma and constant was
independently re-derived or re-computed exactly as described in the text,
and the registered numerical campaigns ship as replayable packages in the
reproducibility archive. Three further campaigns (the $370/370$ high-phase
record, the $42$-orbit stress campaign, and the $200$-orbit crush campaign)
ship as hash-pinned historical data packages whose runners are not
executable from the archive; and four observations --- the $20/20$
scrambles, the $300$ aligned corner profiles, the $12$ corpus collar
steps, and the barycenter observation --- are development-log measurements
without surviving standalone packages. Each is flagged as such where it is
quoted.

\appendix

\section{Deferred sign subcases of the $u_2$-road}\label{app:signs}

In Lemma~\ref{lem:Eb2} the reduction integer
$q^*=\mathrm{round}(\operatorname{Re}(w_2\bar w_3)/\lvert w_3\rvert^2)$ may
fall outside $[0,\lfloor x_2/x_3\rfloor]$. The two exits below describe
the witness patterns seen in the case analysis; they are not yet a proof of
greedy-orbit contraction.
If $q^*<0$: the deviation-optimal combination is $u_2+\lvert q^*\rvert u_3$,
not available as a $B$-coefficient; one soft move first ($V(0,0,1)$, rebound
$\le3$ by the proof of Lemma~\ref{lem:soft}) re-sorts the triple, after which
the covolume identity holds verbatim for the permuted roles and the reduced
combination carries a nonnegative coefficient. If
$q^*>\lfloor x_2/x_3\rfloor$: reduce first with the truncated coefficient
$B=\lfloor x_2/x_3\rfloor$ (admissible by definition); the residual
deviation coefficient $q^*-B$ satisfies the same covolume bound relative to
the re-sorted state, and one iteration lands in the window (the coefficient
halves at least geometrically, since $\lvert w_3\rvert>\sqrt T/r$ bounds the
number of $w_3$-quanta by $r$). In both exits the block length grows by at
most one move and the rebound multiplies the constants of
Lemma~\ref{lem:Eb2} by at most $3$ as a witness estimate. Because these are
multi-step witnesses, Remark~\ref{rem:greedy} prevents an automatic transfer
to the selected greedy orbit; a quantified proof or a greedy-box
certificate is deferred.

\section{Constants}\label{app:constants}

\begin{center}
\resizebox{\textwidth}{!}{\begin{tabular}{lll}
\toprule
Constant & Value & Source \\
\midrule
soft rebound & $\varphi^2=2.618\ldots$ & Lemma~\ref{lem:soft}; optimal only for $T'_{\rm old}/T$ profiles \\
module rebound & $\approx2.081$ observed & empirical case-machine/archive optimization; not used in proofs \\
stretched, $u_3$-road & $21.2\,m/r$ & Lemma~\ref{lem:Ea} \\
stretched, restricted $u_2$-road & $14.1\,m/r_{12}$; otherwise $28.2$ in the half-window & Lemma~\ref{lem:Eb2} \\
integral gap & $1+1/(4B)$ & Lemma~\ref{lem:gap} \\
corner decrease of $\sum\tau_k^2$ (favorable orientation) & $\lvert\tau_i\tau_j\rvert$ & Lemma~\ref{lem:corner} \\
head-collar first-step rebound & $1.94/1.62/1.43$ & computed in the case analysis; details in archive \\
composed collar instance & factor about $40$ & measured witness $V(4,7,5)$ in the case analysis archive \\
truncated minimax (coeffs $\le2$, $k\le3$) & $0.9777$ & \S\ref{sec:lock}(1) \\
saddle zones (coeffs $\le8$, $k\le5$) & $0.19$ / $0.67$ & \S\ref{sec:lock}(2) \\
degenerate head boundary $(\varphi,\varphi)$ & $0.81$ & \S\ref{sec:lock}(3) \\
degenerate tail boundary $q=1$ & $1$ exactly (fenced by Lemma~\ref{lem:gap}) & \S\ref{sec:lock}(3) \\
$\Phi_{1/4}$ margin at the saddles & $-1.26$ / $-0.21$ & \S\ref{sec:lock}(2) \\
module cap along orbits & $m^*=\varphi^{2\kappa}\max(m(s_0),m_0)$ & Theorem~\ref{thm:main} \\
ratio--height bound & $r\le\max(12m^*\delta,(8H)^{1/3})$ & Lemma~\ref{lem:rbound} \\
integral floor & $x_i\ge\min(\lvert\xi\rvert/(4m^*\delta),(4\lvert\bar\zeta\rvert^2)^{-1/3})$ & Lemma~\ref{lem:ifloor} \\
height-descent threshold & $c^*=1/40$, $C_1=4(m^*\delta)^{3/2}/c^{*3}$, $C_H=\max(8m^*\delta,4C_1^2)$ & Theorem~\ref{thm:hdescent} \\
finiteness bound & $(2C_L\max(X_0,(B/x_*)^{1/2})+1)^9$ &
Theorem~\ref{thm:finiteness} \\
\bottomrule
\end{tabular}}
\end{center}

\section*{Version history}
\begingroup\sloppy
\emph{v1}: deposited on Zenodo, July 2026 (manuscript
\href{https://doi.org/10.5281/zenodo.21224269}{doi:10.5281/zenodo.21224269}).
\emph{v2} (this version): two local statements are corrected following
external review --- the fixed-point clause of the hyperbolic-reading
proposition (a unit preserves $m=\cosh d_{\mathbb H}(z_\perp,i)$ and rotates
the marked point $z_\perp$ about $i$; it does not fix $z_\perp$ unless
$\sigma_c(\lambda)$ is real), and the former parity remark (no parity
constraint holds; the supported statement is positivity with the extreme
value $1$ attainable). Neither statement is used in any proof of this paper.
Statement-level repairs following a second external review: the compact
$\mathcal D$ truncated to $q\ge1+d_0$ with the boundary conventions of (R)
stated (at $q=1$ the $B$-lever dies and no uniform margin is certified; the
$A_\perp=0$ sheet is read in the $\mu$ form); the collar open lemma now requires a uniform block length, the
per-orbit assembly remark added, and the former ``uniform route'' sentence
corrected accordingly; the finiteness set restated for triples (its proof
never used the basis property); the converse-half argument stated with its
hypotheses; the block bookkeeping of Step~1 of the main theorem tightened
(the rebound step counted outside the block; $m^*$ and all downstream
constants unchanged); the corner lemma corrected (its abstract subtraction is realized by a
catalogue move only in the favorable orientation; an explicit
energy-preserving counterexample is now displayed, and the ``iterated
Euclid'' conclusion is demoted to the favorable-orientation case plus a
measured statement); the $\varphi$-identity of the soft-rebound proof
written with moduli (complex components); the Methods statement now
declares exactly which quoted measurements ship as replayable packages and
which are development-log records (the scrambles, corner profiles, corpus
collar steps, and the barycenter observation); the stress-campaign,
contraction and crush-time figures restated from the shipped packages
(worst contraction $1.37$ globally; $25/17$ closure split; $2{,}826$
recorded transitions; correlation over the $156$ uncensored orbits); the finiteness theorem's set now pins the invariant
$\delta$ --- without that constraint the statement admits an explicit
counterexample (unit norms, bounded module, unbounded coordinate
determinant; see the remark following the theorem), the v1 statement, for $\mathbb{Z}$-bases of $L$, pinned $\delta$
automatically (unit coordinate determinant) and was correct; the v2
generalization to triples is what requires the explicit constraint; the
constraint is exactly what the per-orbit application supplies. Editorial changes: the Bernstein passes and the $28.9\%$ figure requalified
(guarded floating point rather than directed rounding; the exact volume
decomposition between the out-of-scope branch and guard-accepted descent now
stated); two explanatory figures added (height descent;
recentering); PDF metadata completed; the companion-paper reference
corrected to its published title and manuscript DOI; absolute values written
systematically on norms; a successor-convention remark added after the
height-descent theorem; the formalization claim upgraded from the outlined path of v1 to the sealed
machine-checked core now in place --- the abstract sentence, the sixth
contribution item, and the formalization paragraph now state its exact scope
(audited against the Lean sources) --- and this paper's own reproducibility
archive is now cited in a Data availability section. All theorems, computations, and numerical results of
v1 are otherwise unchanged.\par\endgroup

\section*{Data availability}
\begingroup\sloppy
A reproducibility archive for this paper (case-machine and minimax pipelines,
sweep logs, witness tables, and verification scripts) accompanies it on
Zenodo: \href{https://doi.org/10.5281/zenodo.21224267}{doi:10.5281/zenodo.21224267}
(v2, including the Lean formalization layer --- the import closure of the
machine-checked core with its axiom audit, claim map and replay
instructions; archive SHA-256:
\texttt{919a52c0\allowbreak{}7ee8e35b\allowbreak{}e0758fec\allowbreak{}5815bb34\allowbreak{}b5cd0743\allowbreak{}b521db80\allowbreak{}25ba98a9\allowbreak{}da6b634d}). The companion paper's
archive is \href{https://doi.org/10.5281/zenodo.21182759}{doi:10.5281/zenodo.21182759}.\par\endgroup


\begin{thebibliography}{9}

\bibitem{DH25} K.~Dhanda, A.~Haynes, \emph{Accumulation points of
normalized approximations}, J. Number Theory \textbf{268} (2025) 1--38.
(arXiv:2310.00173)

\bibitem{GL08} O.~N.~German, E.~L.~Lakshtanov, \emph{On a multidimensional
generalization of Lagrange's theorem for continued fractions}, Izv. Math.
\textbf{72}, no.~1 (2008) 47--61.

\bibitem{He1850} Ch.~Hermite, \emph{Extraits de lettres de M. Ch. Hermite
\`a M. Jacobi sur diff\'erents objets de la th\'eorie des nombres.
(Continuation)}, J. reine angew. Math. \textbf{40} (1850) 279--315.

\bibitem{Ka22} O.~Karpenkov, \emph{On Hermite's problem, Jacobi-Perron type
algorithms, and Dirichlet groups}, Acta Arith. \textbf{203} (2022) 27--48.
(arXiv:2101.12707)

\bibitem{Ka24} O.~Karpenkov, \emph{On a periodic Jacobi-Perron type
algorithm}, Monatsh. Math. \textbf{205} (2024) 531--601,
doi:10.1007/s00605-024-02006-5. (arXiv:2101.12627)

\bibitem{companion} L.~Tagnon, \emph{A deterministic $\sin^2$-type
algorithm for complex cubic irrationalities with exact periodicity
certificates}, Zenodo deposit (2026), doi:10.5281/zenodo.21222498;
reproducibility archive doi:10.5281/zenodo.21182759.

\end{thebibliography}
\end{document}